\documentclass{amsart}
\usepackage{hyperref}
\newcommand*{\mailto}[1]{\href{mailto:#1}{\nolinkurl{#1}}}
\newcommand{\arxiv}[1]{\href{http://arxiv.org/abs/#1}{arXiv:#1}}

\newtheorem{theorem}{Theorem}[section]
\newtheorem{definition}[theorem]{Definition}
\newtheorem{proposition}[theorem]{Proposition}
\newtheorem{lemma}[theorem]{Lemma}
\newtheorem{example}[theorem]{Example}
\newtheorem{corollary}[theorem]{Corollary}
\newtheorem{remark}[theorem]{Remark}
\newtheorem{hypot}[theorem]{Hypothesis}

\def\R{\mathbb R}
\def\C{\mathbb C}
\def\N{\mathbb N}
\def\I{\mathrm i}
\def\E{\mathrm e}

\def\gD{\mathfrak{D}}
\def\ga{\mathfrak{a}}
\def\cI{\mathcal{I}}
\def\cA{\mathcal{A}}
\def\cF{\mathcal{F}}
\def\cN{\mathcal{N}}
\def\cH{\mathcal{H}}

\newcommand{\be}{\begin{equation}}
\newcommand{\ee}{\end{equation}}

\def\im{\mathrm{Im}\, }
\def\re{\mathrm{Re}\, }
\def\ti{\tilde}
\def\col{\mathrm{col}}
\def\ess{\mathrm{ess}}
\def\sgn{\mathrm{sgn}\, }
\def\loc{\mathrm{loc}}
\def\dom{\mathrm{dom}}
\def\Span{\mathrm{span}}

\numberwithin{equation}{section}

\begin{document}

%opening
\title[The similarity problem for indefinite spectral problems]{The similarity problem for indefinite Sturm-Liouville operators \\ and the HELP inequality}

\author[A.~Kostenko]{Aleksey Kostenko}
\address{Faculty of Mathematics\\ University of Vienna\\
Nordbergstrasse 15\\ 1090 Wien\\ Austria}

\email{\mailto{duzer80@gmail.com; Oleksiy.Kostenko@univie.ac.at}}

\thanks{{\it The research was funded by the Austrian Science Fund (FWF): M1309--N13}}

\keywords{indefinite Sturm--Liouville problem, the similarity problem, HELP inequality, regularly varying functions, Fokker--Plank equation}
\subjclass[2010]{Primary 34B24; Secondary 26D10, 34L10, 47A10,47A75}

\begin{abstract}
 We study two problems. The first one is the similarity problem for the indefinite Sturm--Liouville operator
\[
A=-(\sgn\, x)\frac{d}{wdx}\frac{d}{rdx}
\]
acting in $L^2_{ w}(-b,b)$. It is assumed that $w,r\in L^1_{\loc}(-b,b)$ are even and positive a.e. on $(-b,b)$.

The second object is the so-called HELP inequality
\[
\left(\int_{0}^b\frac{1}{\tilde{r}}|f'|\, dx\right)^2 \le K^2 \int_{0}^b|f|^2\tilde{w}\,dx\int_{0}^b\Big|\frac{1}{\tilde{w}}\big(\frac{1}{\tilde{r}}f'\big)'\Big|^2\tilde{w}\, dx,
\]
where the coefficients $\tilde{w},\tilde{r}\in L^1_{\loc}[0,b)$ are positive a.e. on $(0,b)$.

Both problems are well understood when the corresponding Sturm--Liouville differential expression is regular. The main objective of the present paper is to give criteria for both the validity of the HELP inequality and the similarity to a self-adjoint operator in the singular case. Namely, we establish new criteria formulated in terms of the behavior of the corresponding Weyl--Titchmarsh $m$-functions at $0$ and at $\infty$. As a biproduct of this result we show that both problems are closely connected. Namely, the operator $A$ is similar to a self-adjoint one precisely if the HELP inequality with $\tilde{w}=r$ and $\tilde{r}=w$ is valid.

Next we characterize the behavior of $m$-functions in terms of coefficients and then these results enable us to reformulate the obtained criteria in terms of coefficients. Finally, we apply these results for the study of the  two--way diffusion equation, also known as the time-independent Fokker--Plank equation.
\end{abstract}

\maketitle

\tableofcontents

\section{Introduction}\label{sec:intro}

Consider the following indefinite spectral problem
\be\label{eq:i_sp}
-(\frac{1}{r(x)}y')'+q(x) y=\lambda\, (\sgn\, x )w(x) y,\quad x\in (-b,b).
\ee 
It is assumed that $w,r,q\in L^1_{\loc}(-b,b)$ and $w(x),r(x)>0$ a.e. on $(-b,b)$, $0<b\le +\infty$. In the limit circle case at $b$ or $-b$ self-adjoint boundary conditions at the corresponding end are assumed. 

In contrast to the classical case when the weight function is definite, i.e., the weight does not change sign on $(-b,b)$, 
the operator $A$ naturally associated with \eqref{eq:i_sp} 
\be\label{eq:op_A}
A=\frac{(\sgn x)}{w(x)}\Big(-\frac{d}{dx}\frac{\, d}{r(x)dx}+q(x)\Big) %, \quad \dom(A)=\gD_{\max},
\ee
is not self-adjoint in $L^2_w(-b,b)$. However, it is known that under certain positivity type assumptions on $q$ (for example, if $q\ge 0$ a.e. on $(-b,b)$), the spectrum of this problem is real and accumulates at both $+\infty$ and $-\infty$ (see, e.g., \cite{Kar2, KarKos}). Then the following problem arises: {\em what kind of basis properties do the (generalized) eigenfunctions of \eqref{eq:i_sp} have?} There are two essentially different cases: 
\begin{itemize}
\item[1)] the problem \eqref{eq:i_sp} is regular, i.e., $b<\infty$ and $w,r,q\in L^1(-b,b)$,
\item[2)] the problem \eqref{eq:i_sp} is singular. 
\end{itemize}

The first case is widely studied and in the case of even coefficients $w,r$ the spectral properties of $A$ are well understood. More precisely, the fist results on eigenvalues and completeness properties of eigenfunctions for regular problems were obtained by Hilbert, B\^ocher, and Richardson over a century ago and then by Kamke in 1930s (for further details and references we refer to recent papers \cite{BF2}, \cite{Kar2}). %, in a series of papers, Kamke proved that eigenfunctions are complete and form a basis in the energy space.

Motivated by various problems arising in physics, scattering and transport theory \cite{Be81, bp, fk, gmp, Kar1, cvdm}, the problem of whether or not the eigenfunctions of \eqref{eq:i_sp} form a Riesz basis of $L^2_{w}(-b,b)$
%\footnote{If the spectrum of the problem \eqref{eq:i_sp} is not purely discrete, then instead of the Riesz basis property of %eigenfunctions, one should consider the similarity of the operator $A$ to a self-adjoint operator} 
 attracted a lot of attention since the mid of seventies of the last century (see e.g. \cite{Be79, Be85, BF, BF2, CFK, CurLan89, CN95, FSh2, F96, F98, FN98, Kar2, KarKos, KKM_09, KM_08, AK06, AK11, AK_12, Par03, Pa2, Pya89, Vol_96} and references therein). The first general sufficient condition for the Riesz basis property was obtained by Beals in \cite{Be85} and later this condition has been extended and generalized by many authors (for a survey we refer to the recent papers \cite{BF2, CFK}, see also \cite{BF, CurLan89, F96}).  

In spite of a considerable activity in the study of the Riesz basis property of eigenfunctions of \eqref{eq:i_sp}, the existence of problems \eqref{eq:i_sp} which do not have the Riesz basis property was established only in 1996 by H. Volkmer \cite{Vol_96}. More precisely, Volkmer \cite{Vol_96} observed  that the inequality
\be\label{eq:iHELP+}
\Big( \int_0^b \frac{1}{w}|f'|dx\Big)^2\le K^2 \int_0^b |f|^2dx\int_0^b\Big|\big(\frac{1}{w}f'\big)'\Big|^2\, dx,\quad (f\in\dom(A_+)),
\ee
is valid, i.e., there is $K>0$ such that \eqref{eq:iHELP+} holds for all $f\in\dom(A_+)$, if the eigenfunctions of \eqref{eq:i_sp} with $q=\bold{0}$, $r=\bold{1}$ and $w\in L^\infty(0,b)$ form a Riesz basis of $L^2_{w}(-b,b)$. Here 
\[
\dom(A_+):=\{f\in L^2(0,b): \, f, \frac{1}{w}f'\in AC[0,b], \ (\frac{1}{w}f')(b)=0,\ (\frac{1}{w}f')'\in L^2(0,b)\}.
\]

 Noting that there are weights such that \eqref{eq:iHELP+} is not valid (moreover, using a Baire category argument, it was noticed in \cite{Vol_96} that \eqref{eq:iHELP+} is not valid in general), Volkmer gave a positive answer to the existence problem. Explicit examples of weights were given later by Fleige \cite{F98},   Abasheeva and Pyatkov \cite{abpyat} (see also \cite{BF2}). 
 
A significant progress in the study of the Riesz basis property for \eqref{eq:i_sp} was made by Parfenov \cite{Par03, Pa2}. Namely, using Pyatkov's interpolation criterion \cite{Pya89}, Parfenov found a necessary and sufficient condition for the Riesz basis property under the assumptions that $q=\bold{0}$, $r=\bold{1}$, and $w\in L^1(-b,b)$ is even \cite[Theorem 6]{Par03}: {\em the eigenfunctions of \eqref{eq:i_sp} form a Riesz basis of $L^2_w(-b,b)$ if and only if the function $W(x)=\int_0^x w\, dt$ is positively increasing at $0$} (for the definition see Appendix \ref{ap:osv}). Notice that the problem on the Riesz basis property for \eqref{eq:i_sp} is still open if the assumption that $w,r$ are even is dropped. The most recent results can be found in \cite{BF2, CFK} (see also references therein). 
 
If we consider the singular problem \eqref{eq:i_sp}, then the situation becomes more complicated. First of all, the spectrum of \eqref{eq:i_sp} is not necessarily discrete and hence one needs to consider the problem of similarity to a self-adjoint operator instead of the Riesz basis property. It was noticed in \cite{CurLan89} that in the case $q\ge 0$ and $0\notin \sigma_{\rm ess}(A)$ the situation is similar to the regular case. Namely, if the operator $A$ is $J$-nonnegative then it admits a spectral function (a family of $J$-orthogonal projections), which might be unbounded only at $0$ and at $\infty$ (see \cite{Lan82}). In this case the corresponding point is called singular. Otherwise, it is called regular.
%Clearly, if $0\notin\sigma_c(A)$, then the spectral function If 

It turns out that the problem of the regularity at $0$ is much more subtle. Namely, first results for the case $0\in\sigma_{\rm ess}(A)$ were obtained only in the mid of 1990s (see \cite{CN95, FN98, FSh2}) and to the best of our knowledge the similarity of $A$ to a self-adjoint operator was established for several particular classes of operators (see \cite{KKM_09, KM_08}). Moreover, the existence of operators $A$ with the singular critical point $0$ was established in \cite{KarKos} (see also \cite[\S 5]{KKM_09}). 

Now let us return to the inequality \eqref{eq:iHELP+}. This inequality is a particular case of the Hardy--Littlewood--Polya--Everitt (HELP) inequality. Namely, the famous Hardy--Littlewood inequality \cite[Chapter VII]{hlp} is a special case of \eqref{eq:iHELP+} with $K=2$, $b=+\infty$, $w=\bold{1}$. 
In the seminal paper \cite{ev72}, W.N. Everitt considered the following integral inequality
\be\label{eq:iHELP}
\Big( \int_0^b (\frac{1}{r}|f'|+q|f|^2)dx\Big)^2\le K^2 \int_0^b |f|^2w\, dx\int_0^b \Big|\frac{1}{w}\big(-(\frac{1}{r}f')'+qf\big)\Big|^2w\, dx,\quad (f\in\gD_{\max}).
\ee  Here $K$ is a positive constant; the coefficients $w,r,q\in L^1_{\loc}[0,b)$ are real valued and $w,r$ are assumed to be positive on $[0,b)$; $\gD_{\max}$ is the maximal linear manifold of functions for which both integrals on the right-hand side of \eqref{eq:iHELP} are finite. 

In \cite{ev72}, Everitt connected the above inequality with the Weyl--Titchmarsh $m$-function of the Sturm--Liouville differential equation
\be\label{eq:iSL}
-(\frac{1}{r(x)}f')'+q(x) f=\lambda\, w(x) y, \quad x\in[0,b).
\ee 
Under the assumptions $b=+\infty$, $w\equiv 1$ on $\R_+$ and $\eqref{eq:iSL}$ is regular at $x=0$ and  strong limit point at $+\infty$, Everitt obtained beautifull necessary and sufficient condition for the validity of the HELP inequality in terms of the  $m$-function associated with \eqref{eq:iSL} (see Theorem \ref{th:ev71} below). Moreover, the best possible value of $K$ and all cases of equality in \eqref{eq:iHELP} are indicated in terms of $m$. The proof in \cite{ev72} follows the line of one of the Hardy--Littlewood proofs \cite{hlp} and of course the analysis of \cite{ev72} extends to a wider setting: for the case of nonconstant  $w$ see \cite{ee82}, the case of a regular endpoint $b$ or, more general, the limit circle case at $b$ is addressed in \cite{Ben87} and \cite{ee91}. Note also that Evans and Zettl \cite{ez78} found a general operator theoretic approach to \eqref{eq:iHELP} (see also \cite{ee82}), which allows to study the inequalities of the type \eqref{eq:iHELP} for other differential and difference operators, operators on trees etc. For further information on HELP type inequalities we refer to \cite{Ben87, be, br08, ee82, ee91} (see also references therein).     

Again, the HELP inequality \eqref{eq:iHELP} is well understood in the regular case. Namely, in \cite{Ben87}, Bennewitz gave a necessary and sufficient condition for the validity of \eqref{eq:iHELP} in terms of coefficients (see Theorem \ref{Bennewitz_HELP}). In particular, {\em the inequality \eqref{eq:iHELP+} is valid if and only if the function $W(.)$ is positively increasing at $0$}. Bennewitz's proof is based on a thorough analysis of the asymptotic behavior of the corresponding Weyl--Titchmarsh $m$-function at $\infty$ (see Section \ref{sec:help} for details). It is interesting to note that the class of weights such that \eqref{eq:iHELP+} is valid coincides with the class of even weights $w$ such that \eqref{eq:i_sp} with $r=\bold{1}$ and $q=\bold{0}$ has a Riesz basis property. Thus, in the regular case, Volkmer's condition is not only necessary but is also sufficient for the Riesz basis property of eigenfunctions of \eqref{eq:i_sp}.  
However, this fact was noticed only recently (see \cite{BF2}). 

As for the HELP inequality in the case of a singular end-point $b$, there are only a few particular results (see \cite{ev72}, \cite{ee82}, \cite{ee91}, \cite{br08} and also references therein). The main difficulty in this case is the local behavior of the Weyl--Titchmarsh function at $0$. 

Our main focus is on the problem \eqref{eq:i_sp} and the inequality \eqref{eq:iHELP} in the singular case. Namely, in the case $q=\bold{0}$, we present criteria for both the validity of \eqref{eq:iHELP} and the similarity of the operator $A$ given by \eqref{eq:op_A} to a self-adjoint one. Note that these criteria can be extended to the case of a non-zero potential $q$ under certain positivity type assumptions. Our main tool is the Weyl--Titchmarsh $m$-function (see Section \ref{ss:a2} for definitions) and the analysis is based on the study of the asymptotic behavior of $m(.)$ at zero and at infinity. Let us mention that the behavior of the $m$-function at $\infty$ is widely studied in the literature (see \cite{Ben89} and also Section \ref{ss:a01}). However, the behavior of $m(.)$ at  finite real points has been investigated only for particular classes of Sturm--Liouville operators (decaying potentials, periodic and quasi-periodic coefficients etc.). It is a surprising fact that for "polar" Sturm--Liouville operators ($q=\bold{0}$) the behavior of the $m$-function at $0$ can be characterized by means of behavior of coefficients $w$, $r$ at the singular end.
%\footnote{At the final stage of preparation of this manuscript, the author found out the paper \cite{kas} by Y. Kasahara, who seems %to be the first } 

Let us describe the content of the paper. In Section \ref{sec:II}, we recall the notion of the Weyl--Titchmarsh $m$-function and describe its main properties. In Subsection \ref{ss:a01}, we review basic results on high-energy asymptotic behavior of the $m$-function. New results are presented in Subsection \ref{ss:a02}. There we describe in terms of coefficients $w$ and $r$ the asymptotic behavior of the $m$-function at $0$. In particular, we show that in the case $q=\bold{0}$ and $w,r\notin L^1(0,b)$, {\em the Neumann $m$-function satisfies}
\be\label{eq:I_06}
\sup_{y\in(0,1)}\frac{\re m(\I y)}{\im m(\I y)}<\infty\qquad \Big(\sup_{y\in(0,1)}\frac{\im m(\I y)}{\re m(\I y)}<\infty\Big)
\ee   
{\em precisely if the function $R\circ W^{-1}$ ($W\circ R^{-1}$) is positively increasing at $\infty$}. Here $W(x)=\int_0^x w\, dt$,  $R(x)=\int_0^xr\, dt$, and $W^{-1}$, $R^{-1}$ denote the corresponding inverse functions. Moreover, Kasahara \cite{kas} showed that it is possible to obtain one term asymptotic formula for $m$ at $0$ if $R\circ W^{-1}$ is a regularly varying function at $\infty$ (see Theorem \ref{th:ben_B}). % (for definition and basic properties see Appendix \ref{ap:osv}).

Section \ref{sec:help} is devoted to the HELP inequality \eqref{eq:iHELP}. Firstly, we establish a new criterion for the validity of \eqref{eq:iHELP} in terms of the $m$-coefficient (see Theorem \ref{th:ak_crit+}): {\em if $q=\bold{0}$, then the inequality \eqref{eq:iHELP} is valid if and only if }
\be\label{eq:I_07}
\sup_{y>0}\frac{\re m(\I y)}{\im m(\I y)}<\infty. 
\ee
In the regular case, this criterion was established in \cite{AK_12} and in contrast to the classical Everitt criterion, this result shows that it suffices to know the behavior of $m(.)$ only along the imaginary semi-axis. Note also that \eqref{eq:I_07} is necessary for the validity of \eqref{eq:iHELP} with $q\neq 0$ without any positivity type assumptions. 

Next, combining \eqref{eq:I_07} with the results from Section \ref{sec:II}, we arrive at the following characterization of coefficient $w$ and $r$, for which \eqref{eq:iHELP} is valid (see Theorem \ref{th:HELP_wr}): {\em  if $q=\bold{0}$ and $w,r\notin L^1(0,b)$, then the HELP inequality \eqref{eq:iHELP} is valid precisely if the function  $R\circ W^{-1}$ is positively increasing at both $0$ and $\infty$}. Also, using the connection between positively increasing functions and regularly varying functions, we obtain simple necessary and sufficient conditions in terms of coefficients. Moreover, in Subsection \ref{ss:III.3} we discuss the case of a non-zero $q$ and using the Liouville transformation we establish a criterion for the validity of \eqref{eq:iHELP} under the assumption that the $m$-function belongs to the Krein--Stieltjes class $(S)$ (see Theorem \ref{th:help_q}).

The indefinite spectral problem \eqref{eq:i_sp} is considered in Section \ref{sec:sim}. Our central result,  Theorem \ref{th:simcr_m}, states that in the case of even coefficients $w,r,q$ {\em the $J$-nonnegative operator $A$ is similar to a self-adjoint one if and only if}
\be\label{eq:I_08}
\sup_{y>0}\frac{\im m_+(\I y)}{\re m_+(\I y)}<\infty. 
\ee
Here $m_+$ is the Neumann $m$-function associated with \eqref{eq:i_sp} on $(0,b)$ (for further details see Subsection \ref{ss:2.01}). Let us note that \eqref{eq:I_08} is necessary even without the $J$-positivity assumption \cite{KarKos} since it is necessary for the linear resolvent growth condition (cf. \eqref{lrg}). Let us also mention that  several necessary and sufficient conditions formulated in terms of $m$-functions have been obtained in \cite{Kar2, KarKos, KKM_09, KM_08}. 

The proof of sufficiency of \eqref{eq:I_08} is based on the Veseli\'c--Akopjan similarity criterion \cite{Ves72, Akop_80} (see Theorem \ref{th:veselic}) and is given in Section \ref{ss:proof}. 

Let us emphasize that condition \eqref{eq:I_08} enables us to improve and to extend a number of known results to the case of a singular end-point $b$. Namely: 

(i) Combining Theorem \ref{th:simcr_m} with the results from Section \ref{sec:II} we obtain the following criterion (see Theorem \ref{th:sim_wr}): {\em if $q=\bold{0}$ and $w,r\notin L^1(-b,b)$ are even, then the operator $A$ is similar to a self-adjoint one precisely if the function $W\circ R^{-1}$ is positively increasing at both $0$ and $\infty$}. In the case of a regular endpoint $b$, this result was established by Parfenov \cite{Par03} using a different approach based on Pyatkov's interpolation criterion \cite{Pya89}. 
In the case when $b$ is singular, the similarity of $A$ was established under a very restrictive assumption on the behavior of $w$ at $\infty$ (see Remark \ref{rem:4.10}). However, the connection between positively increasing and regularly varying functions enables us to obtain simple necessary and sufficient conditions, which substantially improve all previous results. Moreover, using necessary conditions, we obtain a class of $J$-positive operators $A$ with the singular critical point $0$ (see Section \ref{ss:0crit}). Note that, all known examples of Sturm--Liouville operators with the singular critical point $0$ are $J$-nonnegative, that is, $0\in \sigma_p(A)$ (cf. \cite{Kar2, KarKos, KKM_09}).

(ii) Since in the case of even coefficients $w,r,q$ condition \eqref{eq:I_08} holds if $A$ satisfies the linear resolvent growth condition (see \cite{KarKos}), we immediately conclude that the similarity of $A$ to a self-adjoint operator is equivalent to the linear resolvent growth condition (see Theorem \ref{th:lrg=sim}). Moreover, using the connection between \eqref{eq:I_07} and \eqref{eq:I_08} (cf. Lemma \ref{lem:m=1/m}), we show that the similarity of $A$ is further equivalent to the validity of a certain HELP inequality \eqref{eq:iHELP}. %The latter extends \cite[Theorem 7.]{AK_12} to the case of a singular end-point $b$. 

(iii) Using the Liouville transformation, we can extend the above results to the case of a non-zero potential $q$ (see Lemma \ref{lem:sim_q}). However, in this case the similarity depends not only on $w$ and $r$ but also on $q$ since the solution of \eqref{eq:i_sp} with $\lambda=0$ now play a role. Also this shows that in this case the similarity depends not only on a behavior of coefficients at $0$ and  $b$, but also on a local behavior of coefficients on $(-b,b)$. This fact was observed in \cite[\S 5]{KKM_09}. Moreover, Lemma \ref{lem:sim_q} allows us to obtain simple necessary and sufficient conditions. For instance we investigate the similarity of $A$ under the assumption that there are $l\ge -\frac{1}{2}$ and $x_0>0$ such that for $x>x_0$
\be
q(x)=\frac{l(l+1)}{x^2}+\tilde{q}(x),\quad \int_{x_0}^\infty x|\ti{q}(x)|dx<\infty.
\ee 
Note that the case $l=0$ and $w=r=\bold{1}$ was studied in \cite[\S 4]{KKM_09}. However, our approach allows to treat the similarity for general weights and an arbitrary $l\ge -\frac{1}{2}$ (see Lemma \ref{cor:4.13} for the case $w=r=\bold{1}$ and also the proof of Lemma \ref{lem:6.5} for the case $w=x$).

In the final Section \ref{sec:FP} we investigate the well-posedness of boundary value problems for the two-way diffusion equation, also known as the stationary Fokker--Plank equation      
\be\label{eq:I_10}
(\sgn x)w(x)u_t=(\frac{1}{r(x)}u_x)_x-q(x)u,\quad x\in(-b,b),\quad 0<t\le t_0\le \infty.
\ee
Due to the sign change in the left-hand side, this parabolic equation is of "forward--backward" type.  Equation \eqref{eq:I_10} arises in kinetic theory and in the theory of stochastic processes and have a long history \cite{bg}, \cite{Be81}, \cite{bp}, \cite{fk}, \cite{gmp}, \cite{pag1}, \cite{pag2} (see also references therein).  Separation of variables in \eqref{eq:I_10} leads to the spectral problem \eqref{eq:i_sp}
and the well-posedness issue is closely connected with the similarity problem for the operator $A$ (cf. \cite{Be81}, \cite{Be85}, \cite{Kar1}, \cite{Pyat}, \cite{cvdm} and also Theorem \ref{th:karabash} below). The similarity results from Section \ref{sec:sim} allows us to obtain a number of new sufficient conditions for the existence and uniqueness of solutions to boundary value problems for \eqref{eq:I_10} (cf. e.g. Theorem \ref{th:FP}). Let us mention that these conditions substantially extend all previous known conditions (for a comprehensive survey and previous results we refer to \cite{Kar1}).

Appendix \ref{ap:osv} contains necessary definitions and facts on positively increasing functions as well as on Karamata's theory of regularly varying functions. In Appendix \ref{ap:LT}, we present the Liouville transformation which establishes a connection between spectral problems \eqref{eq:i_sp} with $q=\bold{0}$ and a non-zero $q$, under a certain positivity type assumption.

{\bf Notation.} $L^1(a,b)$ and $AC[a,b]$ are the sets of Lebesgue integrable and absolutely continuous functions on a compact interval $[a,b]$; if $w\in L^1_{\loc}(a,b)$ is positive, then $L^2_{w}(a,b)$ stands for the Hilbert space of equivalence classes with the norm $\|f\|=\big(\int_{(a,b)}|f|^2\, w(x)dx\big)^{1/2}$; $L^2(a,b):=L^2_{w}(a,b)$ if $w\equiv 1$.

$\N,\R,\C$ have the standard meanings; $\C_+$ is the open upper half-plane, $\C_+=\{z\in \C:\, \im z>0\}$; $\bar{z}$ is the complex conjugate of $z\in\C$; $\R_+:=[0,+\infty)$ and $\I \R_+=\{\I y:\ y\in\R_+\}$. Also we shall use the following notation $\cI=(-b,b)$, $\cI_+=(0,b)$, $\cI_-=(-b,0)$. 

Prime $'$ denotes the derivative, $'\equiv \frac{d}{dx}$; the subscript $u_x$ denotes the partial derivative, $u_x=\frac{\partial u}{\partial x}$.

The notation '$(x\in X)$' is to be read as 'for all $x$ from the set $X$'.

%\begin{remark}
%5I%n Theorem \ref{lem:vol}, we assume that the interval is symmetric, i.e., $I=[-1,1]$. However, it can be e for 
%\end{remark}

\section{Asymptotic behavior of $m$-functions}\label{sec:II}

\subsection{The Weyl--Titchmarsh $m$-function}\label{ss:a2}
Let the functions $r,w\in L^1_{\loc}[0,b)$, $0<b\le +\infty$, be positive on $[0,b)$. Consider the Sturm--Liouville spectral problem
\begin{align}\label{eq:a01}
&-\big(\frac{1}{r(x)}y'\big)'=\lambda\, w(x)y,\quad x\in (0,b);\\
&\qquad (\frac{1}{r}y')(0)=0,\quad \lim_{x\to b}(\frac{1}{r}y')(x)=0.\label{eq:a02}
\end{align}
If the endpoint $b$ is singular and in the limit point case, that is,  either $w\notin L^1(0,b)$ or $R(x)=\int_0^x r\,dt \notin L^2_w(0,b)$, then the second condition in \eqref{eq:a02} is obsolete and can be dropped. 

Let $c(x,\lambda)$ and $s(x,\lambda)$ be the system of fundamental solutions of \eqref{eq:a01} satisfying 
\be
c(0,\lambda)=(\frac{1}{r}s')(0,\lambda)=1,\quad (\frac{1}{r}c')(0,\lambda)=s(0,\lambda)=0.
\ee
For $\lambda\in\C\setminus\R$ let also $\psi(x,\lambda)$ be {\em the Weyl solution} of \eqref{eq:a01}: 
\be\label{eq:bc_b}
\begin{cases}
\lim_{x\to b}(\frac{1}{r}\psi')(x,\lambda)=0, &  \text{limit circle case at}\ $b$ ,\\
\psi(x,\lambda)\in L^2_w(0,b),&   \text{limit point case at}\ $b$ .
\end{cases}
\ee
%defined by \eqref{eq:weyl_s}. 
{\em The Weyl--Titchmarsh $m$-function} corresponding to the Neumann boundary conditions is then given by
\be\label{eq:a03}
m(\lambda)=-\frac{\psi(0,\lambda)}{(r^{-1}\psi')(0,\lambda)}=\lim_{x\to b}\frac{(r^{-1}s')(x,\lambda)}{(r^{-1}c')(x,\lambda)},\qquad (\lambda\notin\R),
\ee 
or equivalently
\be\label{eq:psi2}
\psi(x,\lambda)=s(x,\lambda)-m(\lambda)c(x,\lambda)\quad \text{satisfies \eqref{eq:bc_b}}.
\ee

Firstly, it is possible (see for details \cite[\S 2]{Ben89}) to assign $m$-functions with all its usual properties to systems of equations on $(0,b)$ defined by 
\be\label{eq:a04}
\begin{cases}
u_1(x) = & u_1(0)+\int_{[0,x)}u_2(t)dR(t),\\
u_2(x)= & u_2(0)-\lambda\int_{[0,x)}u_1(t)dW(t),
\end{cases}
\ee
where $R,W$ are  increasing left-continuous functions  of locally bounded variation on $(0,b)$ normalized by $W(0)=R(0)=0$ and the integrals are interpreted as Lebesgue--Stieltjes integrals. We shall always assume (for details we refer to \cite[\S 2]{Ben89}) the following
\begin{hypot}\label{hyp:RW}
$R$ and $W$ have no discontinuities in common. 
\end{hypot}

Fix a fundamental solution $U(x,\lambda)=\begin{pmatrix} c & s\\ c^{[1]} & s^{[1]}\end{pmatrix}$ of \eqref{eq:a04} satisfying the standard initial condition at $x=0$, $U(0,\lambda)=\begin{pmatrix} 1 & 0\\ 0 &1\end{pmatrix}$. 
Then define the solution $\Psi=\begin{pmatrix} \psi \\ \psi^{[1]} \end{pmatrix}$ such that
\be\label{eq:a05}
\Psi(x,\lambda):=U(x,\lambda)\begin{pmatrix} -m(\lambda) \\ 1 \end{pmatrix},\quad \begin{cases}
\lim_{x\to b}\psi^{[1]}(x,\lambda)=0, &  \text{l. c. case at}\ $b$ ,\\
\psi\in L^2((0,b);dW),&   \text{l. p. case at}\ $b$ .
\end{cases}
\ee
The function $m$ is called {\em the $m$-function} of \eqref{eq:a04} subject to the Neumann boundary conditions. 
Notice that in the case $dR(x)=r(x)dx$ and $dW(x)=w(x)dx$ with positive $r,w\in L^1_{\loc}[0,b)$, the $m$-functions \eqref{eq:a03} and \eqref{eq:a05} coincide. 

Further, applying the Lagrange formula, we get
\be\label{eq:a06}
\int_0^b|\psi(x,\lambda)|^2dW(x)=\int_0^b|s(x,\lambda)-m(\lambda)c(x,\lambda)|^2dW(x)=\frac{\im m(\lambda)}{\im \lambda},\quad (\lambda\notin\R).
\ee
Equality \eqref{eq:a06} means that $m$ is a Herglotz function. Moreover, the function $m$ 
admits the representation
\be\label{eq:m_repr}
m(\lambda)=C+\int_{\R_+}\frac{d\tau(s)}{s-\lambda},\quad \lambda\notin\R_+, \quad C\ge 0,
\ee
where the positive measure $d\tau$, called the spectral measure,  satisfies
\be\label{eq:m_repr2}
\int_{\R_+}\frac{d\tau(s)}{1+s}<\infty. %,\quad \int_{\R_+}d\tau_\pm(s)=\infty.
\ee
In particular, \eqref{eq:m_repr} means that $m$ belongs to the Krein--Stieltjes class $(S)$ (see \cite{KK1}).

Notice also that in the limit circle case equation \eqref{eq:a06} defines {\em the Weyl circle} $C_{\rho}(z_0)=\{z\in\C:\ |z-z_0|=\rho\}$ {\em at} $\lambda$. The center and the radius of  $C_{\rho}(z_0)$ are given by 
\[
z_0(\lambda)=\frac{(s\overline{c}^{[1]} - s^{[1]}\overline{c})(b,\lambda)}{(c\overline{c}^{[1]} - c^{[1]}\overline{c})(b,\lambda)},\qquad \rho(\lambda)=\big(2|\im \lambda| \int_0^b|c(x,\lambda)|^2dW(x)\big)^{-1}.
\]  

%  Finally, let us mention that $m$ admits an analytic continuation to the set $\R\setminus \sigma(A_+)$, where $\sigma(A_+)$ %denotes the spectrum of the problem \eqref{eq:a01}--\eqref{eq:a02}, and the singularities of $m$ are precisely the points of the %spectrum of \eqref{eq:a01}--\eqref{eq:a02}. 

\begin{lemma}\label{lem:m=1/m}
Assume that $W$ and $R$ are left-continuous nondecreasing functions on $[0,b)$ satisfying Hypothesis \ref{hyp:RW}. Let also $m(\cdot)$ be the Neumann $m$-function \eqref{eq:a05} associated with the problem \eqref{eq:a04} and let $\ti{m}(\cdot)$ be the $m$-function for the system
\be\label{eq:a04B}
\begin{cases}
u_1(x) = & u_1(0)+\int_{[0,x)}u_2(t)dW(t),\\
u_2(x)= & u_2(0)-\lambda\int_{[0,x)}u_1(t)dR(t),
\end{cases}
\ee
 subject to the Dirichlet boundary conditions, that is, 
\be
\Psi(x,\lambda):=U(x,\lambda)\begin{pmatrix} 1 \\ m(\lambda) \end{pmatrix},\quad
\begin{cases}
\lim_{x\to b}\psi(x,\lambda)=0, &  \text{\rm l. c. case at}\ $b$ ,\\
\psi\in L^2((0,b);dR),&   \text{\rm l. p. case at}\ $b$ .
\end{cases}
\ee
Then 
\be\label{eq:m=1/m}
m(\lambda)=-\frac{1}{\lambda\ti{m}(\lambda)},\qquad (\lambda\notin\R_+).
\ee
\end{lemma}

\begin{proof}
It suffices to notice that $\col(u_1,u_2)$ solves \eqref{eq:a04} precisely if $\col(-u_2/\lambda,u_1)$ solves \eqref{eq:a04B}.
\end{proof}

\begin{example}\label{ex:a01}
Let $dR(x)=a\delta(x)$ and $dW(x)=dx$, where $a>0$ and $\delta(.)$ is the Dirac $\delta$-function, i.e., $R(x)=a\chi_{(0,1]}(x)$. Then \eqref{eq:a04} becomes
\[
\begin{cases}
u_1(x) = & u_1(0)+au_2(0)\chi_{(0,1]}(x)\\
u_2(x)= & u_2(0)-\lambda (u_1(0)+a u_2(0))x
\end{cases},\quad x\in(0,1).
\]
Therefore, 
\[
U(x,\lambda)=\begin{pmatrix} 1 & a\chi_{(0,1]}(x)\\ -\lambda x & 1-\lambda ax\end{pmatrix}
\]
and hence the $m$-function is given by
\[
m(\lambda)=a-\frac{1}{\lambda},\quad (\lambda\neq 0).
\]
The Weyl circle at $\lambda\in\C_+$ has its center at $z(\lambda)=a+\frac{\I}{2\im \lambda}$ and radius $\rho(\lambda)=\frac{1}{2\im \lambda}$. Note that a real point $x=a$ belongs to this circle for every $\lambda \in \C_+$. The latter is possible only in some degenerate cases. In particular, for systems \eqref{eq:a04} the following is true: {\em if $W=x$ and the Weyl circle of \eqref{eq:a04B}  at some $\lambda\in \C_+$ contains a real point $a\in\R$, then $dR=a\delta$}.   
\end{example}

%%%%%%%%%%%%%%%%%%%
%%%%%%%%%%%%%%%%%%%
\subsection{Asymptotic behavior of $m$ at $\infty$}\label{ss:a01}
Seems the first results on the asymptotic behavior of $m(\cdot)$ at large $\lambda$ were obtained by V.~A.~Marchenko, M.~G.~Krein, and I.~S.~Kac in 1950-s. 
It was observed by I.S. Kac \cite{K71, K73, KK2} that the behavior of $m$ at large $\lambda$ is determined by the behavior of the functions $w$ and $r$ at $x=0$. Y. Kasahara \cite{kas} improved these results and then applied them for the study of limit theorems for generalized diffusion processes. The most complete results on high-energy asymptotics are contained in the excellent survey \cite{Ben89} by Bennewitz. 
 
 Before formulate the next result we need the following definition.

 \begin{definition}\label{def:ginv}
  {\em The generalized inverse} $f^{-1}$ of a nondecreasing and left continuous function $f:(0,b)\to\R_+$  is defined on the convex hull of $f(0,b)$ by $f^{-1}(x) = \inf\{y:\ f(y)\ge x\}$.
 \end{definition}
%Here is the result of Bennewitz.
%%In \cite{Ben87}, Bennewitz found a criterion for the validity of condition \eqref{eq:ev_mf} for some $K>0$ in terms of the weight %function $r$. His result reads as follows. %Another criterion for the HELP inequality \eqref{Volkmer} was obtained by Bennewitz %\cite{Ben87}in terms of weights $w$.

\begin{lemma}[\cite{Ben87}]\label{lem:Bennewitz}
Let $m$ be the m-function defined by \eqref{eq:a05}.
Then
\be\label{eq:m_ass}
|m(\lambda)|=O\big(\im m(\lambda)\big) \quad \text{as}\quad |\lambda| \to +\infty
\ee
 in any nonreal sector (a sector non intersecting the real axis) if  $R\circ W^{-1}$ is positively increasing at $0$.
 
 Conversely, if $R\circ W^{-1}$ is not positively increasing, then 
 \be
 \sup_{y>1}\frac{\re \, m(\I y)}{\im m(\I y)}=+\infty.
 \ee
%
% Moreover, if $S_0\equiv 1$ on $(0,1)$, then for any $\theta\in (0,\pi)$ there exists a strictly increasing sequence %$\{r_j\}_1^\infty$, $r_j\uparrow +\infty$ such that
% \be
% \frac{\im m_w(r_j\E^{\I\theta})}{\re m_w(r_j\E^{\I\theta})} =o(1)\quad\text{as}\quad j\to\infty.
% \ee
\end{lemma}
\begin{remark}\label{rem:3.2}
The first part of Lemma \ref{lem:Bennewitz} was obtained by C. Bennewitz in \cite{Ben87} (see Lemma on p. 344). It is formulated for regular problems, i.e., under the assumptions $b<\infty$ and $w,r\in L^1(0,b)$. However, the result remains valid in the general case (see concluding remarks  at the end of \cite{Ben87}). 

The second part follows from Theorems 3.2 and 3.3 in \cite{AK_12}.
\end{remark}
  
  \begin{corollary}\label{cor:2.1}
  Let $m$ be the m-function defined by \eqref{eq:a05}. Then 
  \be\label{eq:}
 \sup_{y>1}\frac{\im \, m(\I y)}{\re m(\I y)}<+\infty
 \ee
 if and only if the function $W\circ R^{-1}$ is positively increasing at $0$. 
  \end{corollary}
  
  \begin{proof}
  Consider the problem \eqref{eq:a04B} and let $\tilde{m}$ be the corresponding $m$-function.
  Then, by Lemma \ref{lem:m=1/m}, we get
  \[
  \frac{\im m(\I y)}{\re m(\I y)}=\frac{\im\frac{1}{\I y\, \ti{m}(\I y)}}{\re\frac{1}{\I y\, \ti{m}(\I y)}}=\frac{\re \ti{m}(\I y)}{\im \ti{m}(\I y)},\quad y\neq 0.
  \]
  Applying Lemma \ref{lem:Bennewitz} to the problem \eqref{eq:a04B} we complete the proof.
  \end{proof}
  
  %We also need the following result of I.S.~Kac \cite{K71}, Y.~Kasahara \cite{kas}, and C.~Bennewitz \cite{Ben89}.

\begin{definition}\label{def:f} 
Let {\em the function $f$} be the generalized inverse of 
\be\label{eq:Ff}
F(x)=\frac{1}{x(W\circ R^{-1})(x)}.
\ee
\end{definition}

 Note that $f(y)\to 0$ as $y\to +\infty$ and, moreover, $yf(y)\to 0$ as $y\to \infty$ since $xF(x)\to \infty$ as $x\to 0$.
 
As it was noticed by Kasahara \cite{kas} and later by Bennewitz \cite{Ben89}, regularly varying functions (see Appendix \ref{ap:osv}) play an important role in the study of the asymptotic behavior of $m$-functions. %The next result was obtained by Bennewitz. 
The following result is a particular case of Theorem~4.1 from \cite{Ben89} (see also \cite{K71}, \cite{K73}, \cite[Theorem 2]{kas}).

\begin{theorem}[\cite{Ben89, K71, kas}]\label{th:ben}
Assume that the function $R\circ W^{-1}$ is regularly varying at $0$ with index $\alpha\in (0,\infty)$. Then the Neumann $m$-function \eqref{eq:a05} satisfies
\be\label{eq:m_asymp}
m(\mu\rho)= \frac{K_\nu}{(-\mu)^\nu}f(\rho)(1+o(1)) \quad \text{as}\quad \rho\to\infty,
\ee
where
\be\label{eq:nuC}
\nu=\frac{\alpha}{1+\alpha},\quad K_\nu=\frac{\nu^{1-\nu}\Gamma(\nu)}{(1-\nu)^\nu\Gamma(1-\nu)}.
\ee
The estimate holds uniformly for $\mu$ in any compact set of $\C_+$. Here $\Gamma$ is the classical gamma function.

If $R\circ W^{-1}$ varies slowly (rapidly) at $0$, then the asymptotic formula \eqref{eq:m_asymp}--\eqref{eq:nuC} remains true with $\nu = 1$ ($\nu=0$) and $K_\nu=1$.
\end{theorem}

%%%%%%%%%%%%%%%%%%%
%%%%%%%%%%%%%%%%%%%
 \subsection{Asymptotic behavior of $m$ at $0$}\label{ss:a02}
 In contrast to the  high energy asymptotic behavior of $m$, the asymptotic behavior of $m$-functions at finite real points is insufficiently studied. To the best of our knowledge, there are only a few results in this direction. Of course, for regular problems the answer is simple. Namely, for regular problems the $m$-function is meromorphic in $\C$ and hence either $\lambda=x_0$ is a pole of $m(\cdot)$ (and hence $x_0$ is an eigenvalue), or it is a regular point (and hence $x_0$ is in the resolvent set). If the problem is singular, then a singular continuos spectrum may appear and hence the behavior of $m$ at $x_0$ might be very nontrivial. However, it is a surprising fact that for "polar" Sturm--Liouville operators $-\frac{d}{wdx}\frac{d}{rdx}$ the behavior of $m$ at $\lambda=0$ can be characterized in terms of the behavior of its coefficients $w$ and $r$ at a singular end. Seems the first results in this direction were obtained by I.~S.~Kac and M.~G.~Krein \cite{KK58} and then later by Y. Kasahara \cite{kas}. 
 %Namely, in \cite[Theorems 1--3]{KK58} they gave necessary and sufficient conditions in terms of coefficients for    

% In this section we concentrate our attention on spectral problems \eqref{eq:a01}--\eqref{eq:a02}. 

We begin with the following simple result.
 \begin{lemma}\label{lem:b<inf}
 Let $b\le +\infty$ and let $W$ and $R$ be left-continuous nondecreasing functions on $[0,b)$ satisfying Hypothesis \ref{hyp:RW}. Let also $m$ be the Neumann $m$-function \eqref{eq:a05} of the system \eqref{eq:a04}. %Then: 
 \begin{itemize}
\item[(i)] If $W(b-)<\infty$,
then
 \be\label{eq:a12}
 m(\lambda)=-\frac{a}{\lambda}+\tilde{m}(\lambda),
 \ee
 where $\tilde{m}\in (S)$ and, moreover,
 \be\label{eq:a13}
 \lim_{y\downarrow 0}y\, \tilde{m}(\I y)=0,\qquad \lim_{y\downarrow 0}\frac{\re m(\I y)}{\im \, m(\I y)}=0.
 \ee
 \item[(ii)] If $R(b-)<\infty$ and $W(b-)=\infty$, then
% In particular, the following holds true
 \be\label{eq:a14}
 m(\lambda)=a+\tilde{m}(\lambda), \quad a>0,
 \ee
 where $\tilde{m}\in (S^{-1})$ and
 \be\label{eq:a23}
 \lim_{y\downarrow 0}y\, \tilde{m}(\I y)=0,\qquad \lim_{y\downarrow 0}\frac{\re m(\I y)}{\im \, m(\I y)}=+\infty.
 \ee
% \item[(iii)] if $w,r\notin L^1(0,b)$, then FIXME:
 \end{itemize}
 \end{lemma}
 
 \begin{proof}
(i) If $W$ is bounded on $[0,b)$, then  $\lambda=0$ is the eigenvalue of the problem \eqref{eq:a04} subject to the Neumann boundary conditions. Hence we get
 \[
 \lim_{y\downarrow 0}y\, m(\I y)=\I a>0,
 \]
and therefore $m$ admits the representation \eqref{eq:a12}--\eqref{eq:a13}. 
%It remains to note that \eqref{eq:a14} is immediate from \eqref{eq:a12}--\eqref{eq:a13}.

(ii)  Since $\lambda=0$ is the eigenvalue of the problem \eqref{eq:a04B} if $R(b-)<\infty$, using \eqref{eq:m=1/m} we get
 \[
 \lim_{y\downarrow 0}y\,\ti{m}(\I y) =\lim_{y\downarrow 0}y\, \frac{-1}{\I y\, m(\I y)}=\I \ti{a},\quad \ti {a}>0,
 \]
and hence $m$ admits the representation \eqref{eq:a14}--\eqref{eq:a23}. 
 \end{proof}
 
% \begin{lemma}\label{lem:r<inf}
% Let $b=+\infty$ and  $\lim_{x\to +\infty} (R\circ W^{-1})(x)=c<\infty$. Let also $m$ be the $m$-function \eqref{eq:a03} %corresponding to the problem \eqref{eq:a01}--\eqref{eq:a02}. Then 
% \be\label{eq:a22}
% m(\lambda)=a+\tilde{m}(\lambda), \quad a>0,
% \ee
% where $\tilde{m}\in (S^{-1})$ and
% \be\label{eq:a23}
% \lim_{y\downarrow 0}y\, \tilde{m}(\I y)=0.
% \ee
% In particular, the following holds true
% \be\label{eq:a24}
% \lim_{y\downarrow 0}\frac{\re\, m(\I y)}{\im \, m(\I y)}=+\infty.
% \ee
% \end{lemma}
% 
% \begin{proof}
% Since $\lambda=0$ is the eigenvalue of the following problem
% \[
% -y''=\lambda\, r(x)y,\quad x\in\R_+; \quad y'(0)=0,
% \]
 % if $r\in L^1(\R_+)$, we get
% \[
% \lim_{y\downarrow 0}y\, \frac{-1}{\I y\, m(\I y)}=\I a>0,
% \]
%and hence $m$ admits the representation \eqref{eq:a22}--\eqref{eq:a23}. It remains to note that \eqref{eq:a24} is immediate from %\eqref{eq:a12}--\eqref{eq:a13}.
 %\end{proof}
 
 From now on we shall assume that $W$ and $R$ are unbounded on $[0,b)$. Therefore, \eqref{eq:a04} is in the limit point case. 
 
 \begin{theorem}\label{th:m_inf}
  Let  $W$ and $R$ be left-continuous nondecreasing and unbounded functions on $[0,b)$, satisfying Hypothesis \ref{hyp:RW}. Let also $m(.)$ be the Neumann $m$-function \eqref{eq:a05} corresponding to \eqref{eq:a04}. Then 
 \be
 |\re m(\lambda)|=O(\im \, m(\lambda)),\quad |\lambda|\to 0,
 \ee
 in any nonreal sector if there is $t\in (0,1)$ such that
 \be
 S_\infty(t):=\limsup_{x\to +\infty}\frac{(R\circ W^{-1})(xt)}{(R\circ W^{-1})(x)}\neq 1,
 \ee
 or equivalently, if the function $R\circ W^{-1}$ is positively increasing at $\infty$.
 
 Conversely, if $S_\infty\equiv 1$ on $(0,1)$, or equivalently, $R\circ W^{-1}$ is not positively increasing at $\infty$, then 
 \be
 \sup_{y\in(0,1)}\frac{|\re m(\I y)|}{\im m(\I y)}=+\infty.
 \ee 
 \end{theorem}
 
% Before start the proof of Theorem \ref{th:m_inf} we need some preliminary results. 
% \begin{definition}
% Let the function $f$ be the inverse of $F(x):=\frac{1}{xR^{-1}(x)}$, where $R^{-1}$ denotes the inverse of $R$.
% \end{definition}
%Firstly, 

Before start the proof we need some preparatory lemmas.

\begin{lemma}\label{lem:change}
Let  $W$ and $R$ be left-continuous nondecreasing and unbounded functions on $[0,b)$, satisfying Hypothesis \ref{hyp:RW}. 
Let $\tilde{R}:=R\circ W^{-1}$ and $\tilde{m}$ be the $m$-function of the system
\be\label{eq:a04C}
\begin{cases}
u_1(\xi) = & u_1(0)+\int_{[0,\xi)}u_2(t)d\ti{R}(t),\\
u_2(\xi)= & u_2(0)-\lambda\int_{[0,\xi)}u_1(t)dt,
\end{cases} \quad \xi \in\R_+,
\ee
corresponding to the Neumann boundary condition at $x=0$.  
Then 
\be
\ti{m}(\lambda)=m(\lambda),\quad (\lambda\notin\R_+),
\ee
where $m(.)$ is the $m$-function of the problem \eqref{eq:a04}.
\end{lemma} 

\begin{proof}
If $W$ is absolutely continuous and strictly increasing, then the change of independent variable $\xi=W(x)$ in \eqref{eq:a04} transforms \eqref{eq:a04} into \eqref{eq:a04C}. However, the statement remains true in the general case and for further details we refer to \cite[\S 12]{KK2} and \cite{KK58}, where the case $R(x)=x$ was treated.%\marginpar{CHECK!!!}
\end{proof}

We also need the following result (see \cite[Corollary 2.2]{Ben89}).

\begin{lemma}[\cite{Ben89}]\label{lem:a01}
 Let $W_k=W_\infty=x$ for all $k\in \N$. Suppose also that $R_k$ converges to $R_\infty$ pointwise and locally boundedly. Then the solutions $U_k$ converge pointwise and locally boundedly to the solution $U_\infty$. The convergence  is locally uniform in $\lambda$.
 \end{lemma}
 
Let $f$ be defined as the generalized inverse of the function $F:\R_+\to\R_+$ given by \eqref{eq:Ff}. Notice that $f(y)\uparrow +\infty$ as $y\downarrow 0$ since $F(x)\downarrow 0$ as $x\uparrow +\infty$. However, $yf(y)\downarrow 0$ as $y\downarrow 0$ since $xF(x)\downarrow 0$ as $x\uparrow +\infty$. Observe also that the function $\frac{1}{yf(y)}$ is the inverse of $\frac{1}{x(R\circ W^{-1})(x)}$.

 Consider the function
 \be\label{eq:2.30}
 R_s(t):=\frac{(R\circ W^{-1})(st)}{(R\circ W^{-1})(s)}.
 \ee
 For each $s\in (0,W(b))$ the function $R_s$ maps $(0,1)$ into $[0,1]$. Moreover, $R_s$ is increasing on $(0,1)$. By the second Helly theorem, every sequence has a subsequence $s_{k}$ such that $R_k:=R_{s_{k}}$ converges pointwise and boundedly to some increasing function $R_\infty:[0,1]\to [0,1]$.

 Next, let $s_k$ be such that $s_k\uparrow+\infty$ and $R_k$ converges pointwise and boundedly to some increasing function $R_\infty:[0,1]\to [0,1]$. Define the sequence $\rho_k$ as follows 
 \[
 s_k:=\frac{1}{\rho_kf(\rho_k)},\quad k\in\N.
 \]
  Note that $\rho_k\downarrow 0 $ since $s_k\uparrow +\infty$. 
  
  \begin{lemma}\label{lem:a02}
  Let $m$ be the $m$-function defined by \eqref{eq:a05}. Then $\frac{m(\rho_k\mu)}{f(\rho_k)}$ is asymptotically in the Weyl circle at $\lambda=\mu$ of the problem 
  \be\label{eq:a11}
  \begin{cases}
u_1(x) = & u_1(0)+\int_{[0,x)}u_2(t)dR_\infty(t),\\
u_2(x)= & u_2(0)-\lambda\int_{[0,x)}u_1(t)dt,
\end{cases}\qquad x\in(0,1).
  \ee  
  The latter holds uniformly for $\mu$ in any compact set in $\C_+$.  
  \end{lemma}
  \begin{proof}
  By Lemma \ref{lem:change}, we can consider the system \eqref{eq:a04C} instead of \eqref{eq:a04}. 
  Set $\tilde{R}_k(t):=\frac{\tilde{R}(ts_k)}{f(\rho_k)}$ and consider the corresponding system 
  \be
  \begin{cases}
u_{1,k}(x) = & u_{1,k}(0)+\int_{[0,x)}u_{2,k}(t)d\tilde{R}_k(t),\\
u_{2,k}(x)= & u_{2,k}(0)-\lambda\int_{[0,x)}u_{1,k}(t)dt,
\end{cases}\qquad x\in(0,1).
  \ee  
  Then it is straightforward to check that the system of fundamental solutions is given by
  \be
  c_{k}(t,\mu)=c(ts_k,\rho_k\mu),\quad c_{k}^{[1]}(t,\mu)=f(\rho_k)c_k^{[1]}(ts_k,\rho_k\mu),
  \ee
  and 
  \be
  s_{k}(t,\mu)=\frac{1}{f(\rho_k)}s(ts_k,\rho_k\mu),\quad s_{k}^{[1]}(t,\mu)=s_k^{[1]}(ts_k,\rho_k\mu).
  \ee
Therefore, setting $m_k(\mu):=\frac{m(\rho_k\mu)}{f(\rho_k)}$, we get
\be
\int_0^1|s_k(t,\mu)-m_k(\mu)c_k(t,\mu)|^2dt=\int_0^{s_k}\frac{|s(x,\rho_k\mu)-m(\rho_k\mu)c(x,\rho_k\mu)|^2}{f(\rho_k)^2s_k}dx\le \frac{\im m_k(\mu)}{\im \mu}.
\ee
This inequality yields that $m_k(\mu)$ is in the Weyl circle of the $k$-th system at $\lambda=\mu$. Applying Lemma \ref{lem:a01}, we complete the proof. 
%
%(ii) Let $W$ be an arbitrary left-continuous nondecreasing function such that Hypothesis \ref{hyp:RW} is satisfied. 
%If additionally $W$ is absolutely continuous and strictly increasing, then the change of independent variable $\xi=W(x)$ in %\eqref{eq:a04} transforms \eqref{eq:a04} to the case (i) and we are done. 
  \end{proof}
 
  Now we are in position to prove Theorem \ref{th:m_inf}.
 
  \begin{proof}[Proof of Theorem \ref{th:m_inf}]
  Assume the converse, i.e., there is a sequence $\{\rho_k\}_1^\infty$ and $\theta\in (0,\pi)$ such that $\rho_k\to 0$ and $\im m(\rho_k\E^{\I\theta})/m(\rho_k\E^{\I\theta})\to 0$. Set $s_k:=\frac{1}{\rho_kf(\rho_k)}$ and define the function $\tilde{R}_k:(0,1)\to [0,1]$. Then by the Helly theorem one may choose a subsequence of $s_k$ such that $\tilde{R}_k$ converges, pointwise and boundedly, as $k\to +\infty$ along this subsequence. Without loss of generality we can assume that $\tilde{R}_k$ converges to $R_\infty$ as $k\to\infty$. 
  
Further, by Lemma \ref{lem:a02}, $\frac{m(\rho_k\E^{\I\theta})}{f(\rho_k)}$ is asymptotically in the Weyl circle of \eqref{eq:a11}.  On the other hand, $m(\rho_k\E^{\I\theta})$ is asymptotically real and hence the Weyl circle of \eqref{eq:a11} contains a real point. However (see Example \ref{eq:a01}), the latter is possible precisely if $R_\infty(x)=a\chi_{(0,1]}(x)$. By construction, $S_\infty(x)\ge R_\infty(x)$ and hence $S_\infty$ also has a jump at $x=0$. Noting that $S_\infty$ is submultiplicative, i.e., $S_\infty(t_1t_2)\le S_\infty(t_1)S_\infty(t_2)$, and $S_\infty:(0,1)\to [0,1]$, we finally conclude $S_\infty\equiv 1$ on $(0,1)$.     
  \end{proof}
  
  \begin{corollary}\label{cor:2.2}
  Let  $W$ and  $R$ satisfy the assumptions of Theorem \ref{th:m_inf}. Let also $m$ be the Neumann  $m$-function defined by \eqref{eq:a05}. Then 
  \be\label{eq:2.36}
 \sup_{y\in(0,1)}\frac{\im \, m(\I y)}{\re m(\I y)}<+\infty
 \ee
 if and only if the function $W\circ R^{-1}$ is positively increasing at $\infty$. 
  \end{corollary}
  
  The proof is similar to the proof of Corollary \ref{cor:2.1} and we omit it.
  
  Next we present the following analog of Theorem \ref{th:ben}.
  
  \begin{theorem}\label{th:ben_B}
Let $W$ and $R$ satisfy the assumptions of Theorem \ref{th:m_inf}. Assume that the function $R\circ W^{-1}$ is regularly varying at $\infty$ with index $\alpha\in (0,\infty)$. Then any $m$-function corresponding to the Neumann condition at $x=0$ satisfies
\be\label{eq:m_asymp2}
m(\mu\rho)= \frac{K_\nu}{(-\mu)^\nu}f(\rho)(1+o(1)) \quad \text{as}\quad \rho\to0,
\ee
where
\be\label{eq:nuC2}
\nu=\frac{1}{1+\alpha},\quad K_\nu=\frac{\nu^{1-\nu}\Gamma(\nu)}{(1-\nu)^\nu\Gamma(1-\nu)}.
\ee
The estimate holds uniformly for $\mu$ in any compact set of $\C_+$. 
%
%If $R\circ W^{-1}$ varies slowly (rapidly) at $\infty$, then the asymptotic formula \eqref{eq:m_asymp} remains true with $\nu = 1$ %($\nu=0$) and $K_\nu=1$.
\end{theorem}

\begin{proof}
According to Lemma \ref{lem:change}, it suffices to prove the claim in the case $W(x)=x$. Consider the function $R_s$ defined by \eqref{eq:2.30}. Note that $R_s\to t^{\alpha}$ as $s\to +\infty$ pointwise and locally boundedly on $\R_+$. Then consider the the system \eqref{eq:a11} on the interval  $(0,c)$ with $c>0$. Arguing as in the proof of Lemma \ref{lem:a02}, we can show that $m_s(\mu)$ is asymptotically in the Weyl disc for the system \eqref{eq:a04} on $(0,c)$ with $W=x$ and $R=R_\infty=x^\alpha$. Since $c>0$ is arbitrary and the limit equation is limit point at $\infty$, we conclude that $m_s$ converges to the unique $m$-function of the limit system considered on $\R_+$. Finally it suffices to mention that (see \cite[Lemma 4.6]{Ben89}, \cite[Example 1]{kas}, \cite{AK06})
\be
m_\infty(\lambda)=K_\nu (-\lambda)^{-\nu},\quad \lambda\notin\R_+.
\ee
\end{proof}

\begin{remark}
Let us mention that Theorem \ref{th:ben_B}  can be deduced from \cite[Theorem 2]{kas}, where \eqref{eq:m_asymp2} was established for $\mu=-1$. Moreover, Theorem \ref{th:ben_B} can be extended to the cases $\alpha\in \{0,\infty\}$ (see, e.g., \cite{kas}).
\end{remark}

  %We complete this section by two illustrative examples.
  %
  %\begin{example}\label{ex:a02}
  %Let $b=+\infty$, $W(x)=x$ and  
  %\[
  %R(x)=\exp\big\{\int_1^x\frac{\varepsilon(t)}{t}dt\big\},
  %\]
  %where the function $\varepsilon:\R_+\to\R_+$ is nonnegative continuous and such that $\lim_{x\to\infty}\varepsilon(x)=0$.
%Clearly, $R\circ W^{-1}=R$, Moreover, $R$ is a slowly varying at infinity function (see Appendix \ref{ap:osv}) and hence  is not %positively increasing. 
 % Therefore, $S_\infty\equiv 1$ on $(0,1)$ and hence 
  %\[
  %\sup_{y\in(0,1)}\frac{\re\, m(\I y)}{\im \, m(\I y)}=+\infty.
  %\]
%
%For instance, setting 
%\[
%\varepsilon(x)=\frac{1}{\log(x+1)},
%\]
%we get
%\[
%R(x)=\log(x+1).
%\]
 % \end{example}
%\begin{example}\label{ex:a02}
 % Let $b=+\infty$ and $r(x)=\frac{1}{1+x}$, $x\in\R_+$. Then $R(x)=\log(1+x)$, $x\in\R_+$, and hence
 % \[
 % \lim_{x\to 0}\frac{R(xt)}{R(x)}=t.
 % \]
 % Therefore, $S_0\not\equiv 1$ on $(0,1)$ and hence
 % \[
 % \sup_{y>1}\frac{\re\, m(\I y)}{\im \, m(\I y)}<+\infty.
 % \]
 % However,
 % \[
 % \lim_{x\to +\infty}\frac{R(xt)}{R(x)}=\lim_{x\to +\infty}\frac{\log(1+xt)}{\log(1+x)}=\lim_{x\to +\infty}\frac{\log(x)+\log(t)}%{\log(x)}=1.
%  \]
%  Therefore, $S_\infty\equiv 1$ on $(0,1)$ and hence
 %  \[
 % \sup_{y\in(0,1)}\frac{\re\, m(\I y)}{\im \, m(\I y)}=+\infty.
 % \]
 % \end{example}
  %%%%%%%%%%%%%%%%%%%%%%%%%%%%%%%%%%%%
  %%%%%%%%%%%%%%%%%%%%%%%%%%%%%%%%%%%%

%%%%%%%%%%%%%%%%%%%%%%%%%%%%%%%
%%%%%%%%%%%%%%%%%%%%%%%%%%%%%%%
\section{The HELP inequality}\label{sec:help}

\subsection{Everitt's criterion}
Throughout this section we shall assume that $r,w\in L^1_{\loc}(0,b)$ are positive a.e. on $(0,b)$. %In this section we are interested in conditions  
Consider the following inequality
\be\label{eq:help}
\Big(\int_0^b\frac{1}{r}|f'|^2dx\Big)^2\le K^2\int_0^b|f|^2\, wdx\, \int_0^b\frac{1}{w}\big|\big(\frac{1}{r}f'\big)'\big|^2\, dx,\qquad (f\in\dom(A_+)),
\ee
where %$f\in \dom(A_+)$ and 
\be\label{eq:dom+}
\dom(A_+)=\{f\in L^2_w(0,b):\ f,r^{-1}f'\in AC_{\loc}[0,b),\ w^{-1}(r^{-1}f')'\in L^2_w(0,b)\}.
\ee
If the endpoint $x=b$ is regular or the corresponding differential expression is in the limit circle case at $b$, then we shall also assume that functions from $\dom(A_+)$ satisfy the following boundary condition at the right endpoint: $\lim_{x\uparrow b} (r^{-1}f')(x)=0$. 

\begin{definition}
The inequality \eqref{eq:help} is said to be valid if there is $K>0$ such that \eqref{eq:help} holds for all $f\in\dom(A_+)$.
\end{definition}

Firstly, let us remark that in the particular case $b=+\infty$, $w=r\equiv 1$, and $K=4$, the inequality \eqref{eq:help} is the classical Hardy--Littlewood inequality \cite{hlp}. 

Secondly, notice that the left-hand side in \eqref{eq:help} is finite for all $f\in\dom(A_+)$. Indeed, integrating by parts, we get
\[
\int_0^b\frac{1}{r}|f'|^2dx=\lim_{x\uparrow b}\Big((\frac{1}{r}f')\bar{f}\, \big|_{0}^{x} - \int_0^x \frac{1}{w}(\frac{1}{r}f')'\bar{f}\, wdt\Big).
\]
Clearly, in the regular case the right-hand side is always finite. In the singular case, it suffices to notice that $\lim_{x\uparrow b}(\frac{1}{r}f')(x)\bar{f}(x)=0$ for all $f\in\dom(A_+)$. Indeed, in the limit circle case at $b$,  $\lim_{x\uparrow b}f(x)$ exists and is finite for all $f$ from the maximal domain (cf.  \cite[Lemma 2.1]{ee91}) and, moreover, $\lim_{x\uparrow b}(\frac{1}{r}f')(x)=0$ for all $f\in \dom(A_+)$. In the limit point case at $b$, the result follows from \cite[Corollary on p. 199]{kalf74}.  

The following criterion for the validity of \eqref{eq:help} was found by Everitt \cite{ev72} (see also \cite{ee82}, where the regular case was treated).  

\begin{theorem}[Everitt]\label{th:ev71}
Let $m$ be the $m$-function defined by \eqref{eq:a03}. The inequality \eqref{eq:help} is valid if and only if there is $\theta\in(0,\frac{\pi}{2})$ such that
\be\label{eq:ev71}
-\im (\lambda^2\, m(\lambda))\ge 0,\quad (\lambda \in \Gamma_{\theta}),
\ee
where $\Gamma_{\theta}:=\{z\in\C_+: \frac{\re z}{|z|}\in [-\cos\theta,\cos\theta]\}$.

Moreover, the best possible $K$ in \eqref{eq:help} is given by 
\[
K=\frac{1}{\cos \theta_0},\quad \theta_0:=\inf \big\{\theta\in \big(0,\frac{\pi}{2}\big]:\ \eqref{eq:ev71}\, \text{is satisfied}\big\}.
\] 
\end{theorem}

It is a nontrivial task to apply Everitt's criterion and to obtain conditions for the validity of the HELP inequality \eqref{eq:help} in terms of coefficients. However, in the regular case, Bennewitz \cite{Ben87} found a necessary and sufficient condition for \eqref{eq:help} to be valid. 

\begin{theorem}[Bennewitz]\label{Bennewitz_HELP}
Assume that the end-point $b$ is regular. Then the inequality \eqref{eq:help} is valid if and only if the function $R\circ W^{-1}$ is positively increasing at $x=0$.
\end{theorem}

Let us note that the proof of Theorem \ref{Bennewitz_HELP} is based on Lemma \ref{lem:Bennewitz}.
%\begin{remark}
%Note that Everitt's criterion is true under much more general settings..
%\end{remark}

\subsection{Yet another criterion}
Everitt's criterion for the validity of \eqref{eq:help} requires the knowledge of asymptotic behavior of the corresponding $m$-function $m$ at least in some sector of $\C_+$, which contains the imaginary semi-axis $\I\R_+$. Our main aim is to show that it suffices to know only the behavior of $m$ along the ray $\I\R_+$.
 \begin{theorem}\label{th:ak_crit+}
 Let $m$ be the $m$-function defined by \eqref{eq:weyl_s}. Then the inequality \eqref{eq:help} is valid if and only if
 \be\label{eq:crit+}
 \sup_{y>0}\frac{\re m(\I y)}{\im m(\I y)}<\infty.
 \ee
 \end{theorem}
 
 Before proving Theorem \ref{th:ak_crit+} we need the following result.
 \begin{lemma}\label{lem:m_assympt}
Assume that \eqref{eq:help} is not valid. Then there is a sequence $\{\lambda_j\}\subset\C_+$ such that 
\be\label{eq:lam_01}
\lambda_j=(k_j+\I)y_j,\quad k_j\to +0, \, \, \text{either} \, \, y_j\to+\infty\, \, \text{or}\, \, y_j\to +0,
\ee
and $\arg m(\lambda_j)=o(1)$ as $ j\to+\infty$, i.e.
\be\label{eq:lam_02}
\frac{\im m(\lambda_j)}{\re m(\lambda_j)}=o(1),\quad j\to\infty.
\ee
\end{lemma}

\begin{proof}
Let $\lambda=\rho\E^{\I\theta}\in\C_+$. Denote also $m(\lambda)=|m(\lambda)|\E^{\I \theta_m}$. Note that $\theta_m\in(0,\pi)$ if $\lambda\in\C_+$ since $m$ is Herglotz. Then
\begin{align*}
\im (\lambda^2 m(\lambda))
%&\\
%&=\rho|m_w|(\cos 2\varphi\sin\varphi_m+\sin 2\varphi\cos\varphi_m)\\
%&
=\rho^2|m|\sin( 2\theta+\theta_m).
\end{align*}
The integral representation \eqref{eq:m_repr} implies that $\re m_+(\lambda)>0$ if $\re\lambda\le 0$, i.e., $\theta\in [\frac{\pi}{2},\pi]$. Thus we conclude $\im(\lambda^2m_+(\lambda))<0$ if $\theta\in [\frac{\pi}{2},\frac{3\pi}{4})$. Therefore, if \eqref{eq:help} is not valid, then, by Theorem \ref{th:ev71}, there are sequences $\{\theta_j\}_1^\infty \subset (0,\frac{\pi}{2})$ and $\{\rho_j\}_1^\infty\subset \R_+$ such that $\theta_j\uparrow \frac{\pi}{2}$ and $\im (\lambda_j^2m(\lambda_j))>0$, where $\lambda_j:=\rho_j\E^{\I \theta_j}$, $(j\in\N)$. The latter means that $2\theta_j+(\theta_m)_j<\pi$, where $(\theta_m)_j:=\arg m(\rho_j\E^{\I\theta_j})\in (0,\pi)$. Therefore, $(\theta_m)_j\downarrow 0$ as $j\to \infty$. 

To complete the proof it remains to note that $\lambda_j$ can accumulate only at $0$ or at $\infty$ since $m$ is Herglotz. 
\end{proof}

 \begin{proof}[Proof of Theorem \ref{th:ak_crit+}]
 {\em Necessity.} Assume that  \eqref{eq:help} is valid. Firstly, note that the Weyl solution $\psi(x,\lambda)$ defined by \eqref{eq:psi2} belongs to $\dom(A_+)$. Using \eqref{eq:a06} and \eqref{eq:psi2} we get
 \begin{eqnarray*}
 & \int_0^b |\psi(x,\I y)|^2\, wdx=\frac{1}{y}\im m(\I y),\\
 & \int_0^b \frac{1}{w}\big|\big(\frac{1}{r}\psi'(x,\I y)\big)'\big|^2dx=y^2\int_0^b |\psi(x,\I y)|^2\, wdx=y\im m(\I y),\\
 &\int_0^b \frac{1}{r}|\psi'(x,\I y)|^2dx=\psi(x,\I y)\big(r^{-1}\psi'(x,-\I y)\big)|_{x=0}^b-\I y\int_0^b|\psi(x,\I y)|^2\ wdx\\
 &=m(\I y)-\I \im m(\I y)=\re m(\I y).
 \end{eqnarray*}
 Therefore, substituting $\psi(x,\I y)$ into \eqref{eq:help}, we arrive at
 \[
 \re m(\I y) \le K\im  m(\I y),\quad (y>0).
 \]
 
 {\em Sufficiency.} Assume the converse, i.e., \eqref{eq:help} is not valid. Then,  by Lemma \ref{lem:m_assympt}, there is a sequence $\{\lambda_j\}\subset \C_+$ with the properties \eqref{eq:lam_01}--\eqref{eq:lam_02}.
  
  Using \eqref{eq:m_repr}, observe that for $\lambda_j=x_j+\I y_j=(k_j+\I)y_j$
\[
\im m(\lambda_j)-\im m(\I y_j)=\int_{\R_+}\frac{2sx_j - x_j^2}{s^2+y_j^2}\frac{y_j}{(s-x_j)^2+y_j^2}d\tau(s).
\]
Since
\[
\frac{|2sx_j-x_j^2|}{s^2+y_j^2}\le \frac{2sx_j+x_j^2}{s^2+y_j^2}\le \frac{x_j}{y_j}+\frac{x_j^2}{y_j^2}=k_j+k_j^2\le 2k_j,\quad (k_j\le1),
\]
we get
\be\label{eq:5.07}
\big|\im m(\lambda_j)-\im m(\I y_j)\big|\le 2k_j\im m(\lambda_j),\quad (k_j\le1).
\ee
Further,  
\[ %be\label{eq:5.08}
\re m(\lambda_j)-\re m(\I y_j)=
\int_{\R_+}\Big(\frac{s-x_j}{(s-x_j)^2+y_j^2}-\frac{s}{s^2+y_j^2}\Big)d\tau(s).
\]%ee
Note that
\[
\Big|\frac{s-x_j}{(s-x_j)^2+y_j^2}-\frac{s}{s^2+y_j^2}\Big|\le \frac{s^2 x_j +sx_j^2 +x_jy_j^2}{(s^2+y_j^2)((s-x_j)^2+y_j^2)}\le (2k_j+k_j^2)\frac{y_j}{(s-x_j)^2+y_j^2}.
\]
Thus,  we get
\be\label{eq:5.08}
\big|\re m(\lambda_j)-\re m(\I y_j)\big|\le 
3k_j \im m(\lambda_j),\quad (k_j\le1).
\ee
Therefore, combining \eqref{eq:5.07}, \eqref{eq:5.08} with \eqref{eq:lam_02} and noting that $k_j\downarrow 0$, we obtain
\[
\im m(\I y_j)=o\big(\re m(\I y_j)\big),\quad j\to\infty.
\]
Therefore, \eqref{eq:crit+} is not satisfied. 
The proof is completed.
 \end{proof}
 \begin{remark}
 According to the proof of necessity of \eqref{eq:crit+} for the validity of \eqref{eq:help}, Theorem \ref{th:ak_crit+} means that it suffices to check \eqref{eq:help} on the Weyl solutions corresponding to imaginary $\lambda=\I y$, $(y>0)$.  That is, {\em \eqref{eq:help} is valid if and only if there is $K>0$ such that \eqref{eq:help} holds true for all $f=\psi(x,\I y)$, $y>0$}.
 \end{remark}
 
 Now combining Theorem \ref{th:ak_crit+} with the results from Sections \ref{ss:a01}--\ref{ss:a02}, we arrive at the following characterization of weights $w$ and $r$ for which the HELP inequality is valid.
 
\begin{theorem}\label{th:HELP_wr}
Assume that $w,r\in L^1_{\loc}[0,b)$ are positive a.e. and let $W,R$ be the corresponding distribution functions, $W=\int_0^x wdt$, $R=\int_0^x rdt$. Assume also that the end-point $x=b$ is singular. Then:
\begin{itemize}
\item[(i)] If $w\in L^1(0,b)$, then the inequality \eqref{eq:help} is valid if and only if $R\circ W^{-1}$ is positively increasing at $0$.
\item[(ii)] If $r\in L^1(0,b)$ and $w\notin L^1(0,b)$, then the inequality \eqref{eq:help} is not valid.
\item[(iii)] If  $w,r\notin L^1(0,b)$, then the inequality \eqref{eq:help} is valid if and only if the function $R\circ W^{-1}$ is positively increasing at both $0$ and $\infty$. 
\end{itemize}
\end{theorem}

\begin{proof}
Combining Theorem \ref{th:ak_crit+} with Lemma \ref{lem:b<inf} we prove (i) and (ii). (iii) follows by combing Theorem \ref{th:ak_crit+} with Lemma \ref{lem:Bennewitz} and Theorem \ref{th:m_inf}.
\end{proof}

\begin{corollary}\label{cor:3.1}
Assume that $w,r\notin L^1(0,b)$ are positive a.e. and let $W,R$ be the corresponding distribution functions. 
\begin{itemize}
\item[(i)] if $R\circ W^{-1}$  varies slowly either at $0$ or at $\infty$, then the HELP inequality \eqref{eq:help} is not valid.
\item[(ii)] if $R\circ W^{-1}$ is a regularly varying function at $0$ ($\infty$) with index $\alpha>0$, then the constant $K$  in \eqref{eq:help} satisfies
\[
K\ge \frac{1}{\cos\theta_0},\quad \theta_0=\frac{\pi}{2+\alpha}.
\]
\end{itemize}
\end{corollary}

\begin{proof}
(i) Since a slowly varying function is not positively increasing  (see Appendix \ref{ap:osv}), Theorem \ref{th:HELP_wr}(iii) proves the claim.

(ii) Let $R\circ W^{-1}$ be regularly varying at $0$ with index $\alpha$. Then for each $\theta\in (0,\pi)$, by Theorem \ref{th:ben},
\[
m(\rho\E^{\I \theta})=K_\nu\E^{\I(\pi-\theta)\nu}f(\rho)(1+o(1)),\quad \rho\to \infty.
\]
Therefore, for $\lambda_\theta=\rho\E^{\I\theta}$ we obtain
\[
\im (\lambda_\theta^2 m(\lambda_\theta))=K_\nu \rho^2f(\rho)\sin (2\theta+\pi\nu-\theta\nu)(1+o(1)),\quad \rho\to\infty.
\]
Since for $\theta\in (0,\pi)$, 
\[
\sin (2\theta+\pi\nu-\theta\nu)\le 0 \quad \Leftrightarrow \quad \theta \in [\frac{\pi}{2+\alpha},\pi),
\]
Theorem \ref{th:ev71} completes the proof.
\end{proof}

\begin{corollary}
Assume that $w=\bold{ 1}$, $r\in L^1_{\loc}(\R_+)$ and $r\notin L^1(\R_+)$ is positive a.e. on $\R_+$.
\begin{itemize}
\item[(i)] if there is $x_0>0$ such that $r(x)=\frac{l(x)}{x}$ for $x\ge x_0$, where the function $l\in L^1_{\loc}(\R_+)$ is positive on $(x_0,+\infty)$  and  slowly varying at $\infty$, then the HELP inequality \eqref{eq:help} is not valid.
\item[(ii)] if $r$ is a regularly varying function both at $0$ and $\infty$ with indexes $\alpha_0,\alpha_\infty>-1$, respectively, then the HELP inequality \eqref{eq:help} is valid and the constant $K$  in \eqref{eq:help} satisfies
\[
K\ge \max\{\frac{1}{\cos\theta_0},\frac{1}{\cos\theta_\infty}\},\quad \theta_i=\frac{\pi}{3+\alpha_i},\quad i\in\{0,\infty\}.
\]
\end{itemize}
\end{corollary}

\begin{proof}
(i) Note that for $x\ge x_0$
\[
R(x)=\int_0^x r(t)dt=R(x_0)+\int_{x_0}^x \frac{l(t)}{t}dt=R(x_0)+\tilde{R}(x).
\]
Since $r\notin L^1(\R_+)$, we get $\tilde{R}\to \infty$ as $x\to \infty$. Moreover, 
by Theorem \ref{th:karamata}(ii), $\tilde{R}$ is slowly varying at $\infty$.
Corollary \ref{cor:3.1}(i) proves the claim.

(ii) Since $r$ is regularly varying at $\infty$ with index $\alpha_\infty>-1$, 
%$r(x)=\tilde{r}(x)x^{\alpha_\infty}$ for $x>1$, where $\tilde{r}$ is slowly varying at $\infty$. Then 
by Theorem \ref{th:karamata} (formula \eqref{eq:a.4}), we get
\[
R(x)=\int_0^x r(t)dt=C+\int_1^\infty r(t)dt\sim \frac{x}{1+\alpha_\infty}r(x),\quad x\to \infty.
\]
Therefore, $R$ is regularly varying at $\infty$ with index $1+\alpha_\infty$. 

Similarly, since $r$ is regularly varying at $0$ with index $\alpha$, the function $r(1/x)$ varies regularly at $\infty$ with index $-\alpha_0$. Therefore, as $x\to 0$, we obtain by formula \eqref{eq:a.5}
\[
R(x)=\int_0^x r(t)dt=\int_{1/x}^\infty \frac{r(1/t)}{t^2}dt\sim \frac{x}{1+\alpha_0}r(x),
\]
and hence $R$ varies regularly at $0$ with index $1+\alpha_0$. Corollary \ref{cor:3.1} completes the proof.
\end{proof}
%Note that Karamata's characterization theorem, Theorem \ref{th:karamata}, enables us to reformulate Corollary \ref{cor:3.1} in %terms of $r$ and $w$ 

\begin{remark}
Let us show that the connection between the HELP inequality \eqref{eq:help} and the properties of functions $w$ and $r$ can be observed in a straightforward manner. Note that in the case of a regular end point $b$ this connection was first observed by Abasheeva and Pyatkov in \cite{abpyat}, where they generalized Fleige's example \cite{F98}. %We follow their considerations. 

Assume for simplicity $w\equiv 1$ on $(0,b)$.  Let $a_n, b_n\in (0,b)$ satisfy $0<a_n<b_n<b$. Define the function $f_n:(0,b)\to\R_+$ as follows
\be\label{eq:f_n}
f_n(x):=\int_x^b r(t)h_n(t)dt,\qquad h_n(x):=\begin{cases}
1, & x\in(0,a_n)\\
\frac{x-b_n}{a_n-b_n}, & x\in(a_n,b_n)\\
0, & x>b_n
\end{cases}.
\ee
Clearly, $f\in L^2(0,b)$ and $f\in \dom(A_+)$ since %Note that
$
\frac{1}{r}f_n'=-h_n$ and $\big(\frac{1}{r}f_n'\big)'=\frac{1}{b_n-a_n}\chi_{(a_n,b_n)}$. Let us denote 
\be
A_n:=R(a_n)=\int_0^{a_n}r(x)dx,\quad B_n:=R(b_n)=\int_0^{b_n}r(x)dx.
\ee
Thus we get
\[
I_2(n):=\int_0^b\big|\big(\frac{1}{r}f_n'\big)'\big|^2dx=\frac{1}{b_n-a_n},
\]
\[
I_1(n):=\int_0^b\frac{1}{r}|f_n'|^2dx=\int_0^br|h_n|^2dx\ge \int_0^{a_n}r(x)dx=A_n,
\]
and moreover 
\[
I_0(n):=\int_0^b|f_n|^2dx\le (b_n-a_n)(B_n-A_n)^2+a_nB_n^2.
%=\int_{a_n}^{b_n}\Big|\int_x^{b_n}r(t)\frac{t-b_n}{a_n-b_n}dt\Big|^2dx+\int_0^{a_n}\Big|C_n+\int_x^{a_n}r(t)dt\Big|^2dx,
\]
%where
%\[%
%C_n=\int_{a_n}^{b_n}r(t)\frac{t-b_n}{a_n-b_n}dt.
%\]
%Denote also
%\be
%A_n:=\int_0^{a_n}r(x)dx,\quad B_n:=\int_0^{b_n}r(x)dx.
%\ee
%Then one gets the following estimates
%\[
%A_n\le I_1(n),\quad I_0\le (b_n-a_n)(B_n-A_n)^2+a_nB_n^2.
%\]
Plugging this into \eqref{eq:help}, we arrive at the following estimate for the constant $K$:
\[
\frac{(b_n-a_n)A_n^2}{(b_n-a_n)(B_n-A_n)^2+a_nB_n^2}\le K,
\]
and hence
\be\label{eq:K_n}
\frac{1}{K}\le (\frac{B_n}{A_n}-1)^2+\frac{a_n}{b_n-a_n}(\frac{A_n}{B_n})^2,\quad (n\in\N).
\ee
Therefore, we conclude that {\em 
if there are sequences $\{a_n\}_1^\infty,\{b_n\}_1^\infty\subset (0,b)$ such that
\be\label{eq:nec_n}
a_n<b_n,\quad \frac{a_n}{b_n}\to 0,\quad \text{and}\quad \frac{R(a_n)}{R(b_n)}\to 1, 
\ee
then \eqref{eq:help} is not valid}.

However, by Lemma \ref{lem:osv}, the latter means that the property of the function $R$ to be  positively increasing at $0$ and at $\infty$ (of course, if $b=+\infty$) is necessary for the HELP inequality \eqref{eq:help}. On the other hand, Theorem \ref{th:HELP_wr} states that the family of test functions \eqref{eq:f_n} is sufficient to check the validity of the HELP inequality \eqref{eq:help}. 
\end{remark}

\subsection{The general case: concluding remarks}\label{ss:III.3}

Consider the general Sturm--Liouville expression $\ell=\frac{1}{w}\big(-\frac{d}{dx}\frac{d}{rdx}+q\big)$ and the corresponding  HELP inequality
\be\label{eq:help_gen}
\Big(\int_0^b\big(\frac{1}{r}|f'|^2+q|f|^2\big)dx\Big)^2\le K^2\int_0^b|f|^2\, w\, dx\, \int_0^b\big|-\big(\frac{1}{r}f'\big)'+qf\big|^2\, w\, dx,
\ee
where $f\in \dom(A_+)$ and 
\be\label{eq:dom_gen}
\dom(A_+)=\{f\in L^2_w(0,b):\ f,r^{-1}f'\in AC_{\loc}[0,b),\ \ell[f]\in L^2_w(0,b)\}.
\ee

\begin{remark}\label{rem:3.11}
As it was mentioned, the left-hand side in \eqref{eq:help_gen} is finite for all $f\in\dom(A_+)$ if $q=\bold{0}$. However, for nonzero $q$ it might happen that the left-hand side is infinite for some $f\in \dom(A_+)$ even if the minimal operator associated with $\ell$ is lower semibounded (see, e.g. \cite{kalf74}). Therefore, in what follows we either assume that {\em $\ell$ is strong limit point at $b$}, that is $\lim_{x\to b}(\frac{1}{r}f')(x)\bar{f}(x)=0$ for all $f\in\dom(A_+)$ and hence $\lim_{x\to b}\int_0^x\big(\frac{1}{r}|f'|^2+q|f|^2\big)dt<\infty$ (see \cite{ev76}) or we shall understood the left-hand side in \eqref{eq:help_gen} as {\em the generalized Dirichlet form} $D[f]$ (see \cite{Lan05}):
\be\label{eq:dform}
D[f]:=\int_0^b \ell[f]\bar{f}\, w\, dx-(\frac{1}{r}f')(0)\bar{f}(0),
\ee 
which is clearly finite for all $f\in\dom(A_+)$. In this case, \eqref{eq:help_gen} reads as follows
\be\label{eq:help_mod}
(D[f])^2\le K^2\, \|f\|^2_{L^2_w}\, \|\ell[f]\|^2_{L^2_w},\quad (f\in \dom(A_+)).
\ee
The analogue of Everitt's criterion for the validity of \eqref{eq:help_mod} was established in \cite{Lan05}.
\end{remark}

 Assume that the $m$-function associated with $\ell$ and the Neumann boundary condition at $x=0$ belongs to the Krein--Stieltjes class $(S)$. 
Firstly, notice that Theorem \ref{th:ak_crit+} remains valid in this case (in the sense described in Remark \ref{rem:3.11}, cf. \cite[Theorem 6.1]{ee82}, \cite[Theorem 3.1]{Lan05}). Moreover, if the endpoint $b$ is regular, then  Bennewitz's Theorem remains true after a minor modification: {\em the inequality \eqref{eq:help_gen} is valid if and only if the function $R\circ W^{-1}$ is positively increasing at $0$ and $\lambda=0$ is a pole of $m$ (or equivalently, $0$ is the eigenvalue of the Neumann problem)}. 

However, if the endpoint $b$ is singular, then Theorem \ref{th:HELP_wr} is no longer true. Namely, since the behavior of the $m$-function at infinity depends on the behavior of $R\circ W^{-1}$ at $0$ (cf. \cite{Ben89}), the condition {\em $R\circ W^{-1}$ is positively increasing at $0$} is necessary for the validity of \eqref{eq:help}. However, the condition {\em $R\circ W^{-1}$ is positively increasing at $\infty$} is neither necessary nor sufficient since the behavior of $m(\lambda)$ at $0$ depends not only on the behavior of the potential $q$ at a singular end, but also on its local behavior on $(0,b)$. 

To demonstrate this let us consider a particular case of \eqref{eq:help_gen} assuming $r=\bold{1}$ and $q\ge 0$. 
Since $q$ is nonnegative, the fundamental solutions $c(x,0)$ and $s(x,0)$ of $-y''+qy=0$ are positive on $\R_+$. Let us transform the operator $-\frac{d^2}{dx^2}+q(x)$ to a Krein string operator by using the Liouville transformation from Appendix \ref{ap:LT}. 

Namely, define the following functions $w:[0,B)\to \R_+$
\be\label{eq:3.16}
\ti{w}(\xi)=w(x)c^4(x,0),\quad \xi(x)=\int_0^x\frac{dt}{c^2(t,0)},\quad B=\lim_{x\to+\infty}\xi(x).
%\tilde{W}(\xi):=\int_0^{x(\xi)} c^2(t,0)\, w(t)dt, 
\ee
Denote by $x[\xi]$ the inverse of $\xi(\cdot)$. Then 
\be\label{eq:3.17}
\ti{W}(\xi)=\int_0^\xi \ti{w}(\mu)d\mu=\int_0^{x} c^2(t,0)w(t)dt,\quad x\in\R_+.
\ee

Next observe that by Proposition \ref{prop:LT} and Theorem \ref{th:ev71} the HELP inequality \eqref{eq:help_gen} is valid precisely if the following HELP inequality is valid with the same constant $K$ 
\be\label{eq:help_wLT}
\Big(\int_0^B|f'|^2\, dx\Big)^2\le K^2\int_0^B|f|^2\, \ti{w}\, dx\, \int_0^B\frac{1}{\ti{w}}\big|f''|^2\, dx.
\ee
Note that the end-point $B$ might be regular. Namely, this is the case if $B<\infty$ and $W(B)<\infty$, or equivalently, $\frac{1}{c(.,0)}\in L^2(\R_+)$ and $c(.,0)\in L^2_w(\R_+)$.
%Therefore, there are only three possibilities:
%\begin{itemize}
%\item[1)] $B<\infty$, $w\notin L^1(0,B)$, or equivalently, $\frac{1}{c(.,0)}\in L^2(\R_+)$.
%\item[2)] $B=\infty$, $w\in L^1(\R_+)$, or equivalently, $c(.,0)\in L^2(\R_+)$, $\frac{1}{c(.,0)}\notin L^2(\R_+)$.
%\item[3)] $B=\infty$, $w\notin L^1(\R_+)$,  or equivalently, $c(.,0),\frac{1}{c(.,0)}\notin L^2(\R_+)$.
%\end{itemize}

Now, applying Theorem \ref{th:HELP_wr} and Theorem \ref{Bennewitz_HELP}, we arrive at the following 
\begin{theorem}\label{th:help_q}
Let $r=\bold{1}$ and $q\in L^1_{\loc}(\R_+)$ be nonnegative on $\R_+$. Let also $c(x,0)$ be the solution of $-y''+q(x)y=0$ such that $c(0,0)=1$ and $c'(0,0)=0$. %the functions $w, W, x$ be defined by \eqref{}. 
Then:
\begin{itemize}
\item[(i)] if $c(.,0)\in L^2_w(\R_+)$, then the HELP inequality \eqref{eq:help_gen} is valid if and only if $W^{-1}$ is positively increasing at $0$ ,
\item[(ii)] if $\frac{1}{c(.,0)}\in L^2(\R_+)$ and $c(.,0)\notin L^2_w(\R_+)$, then the HELP inequality \eqref{eq:help_gen} is not valid,
\item[(iii)] if $\frac{1}{c(.,0)}\notin L^2(\R_+)$ and $c(.,0)\notin L^2_w(\R_+)$, then 
 the HELP inequality \eqref{eq:help_gen} is valid if and only if the function $\ti{W}^{-1}(\xi)$ given by \eqref{eq:3.16}, \eqref{eq:3.17} is positively increasing at infinity.
\end{itemize}
\end{theorem}

Next let us give a simple proof of one result of W.N. Everitt \cite[\S 15]{ev72}. Consider the inequality
\be\label{eq:help_wr=1}
\Big(\int_0^{+\infty}\big(|f'|^2+q|f|^2\big)dx\Big)^2\le K^2\int_0^{+\infty}|f|^2\, dx\, \int_0^{+\infty}\big|-f''+qf\big|^2\, dx,\quad (f\in \dom(A_+)).
\ee

Note that $W(x)=R(x)=(R\circ W^{-1})(x)=x$ in this case and hence $R\circ W^{-1}$ is positively increasing at both $0$ and $\infty$.

\begin{corollary}[\cite{ev72}]\label{cor:3.11}
Let $q\in L^1_{\loc}(\R_+)$ be nonnegative, $q\ge 0$. Then the HELP inequality \eqref{eq:help_wr=1} is valid if and only if $q\equiv 0$.
\end{corollary}

\begin{proof}
In the case $q=0$, \eqref{eq:help_wr=1} is the classical Hardy--Littlewood inequality, which is valid with $K=2$.

Assume now that $q> 0$ on a set of a positive Lebesgue measure $E\subseteq \R_+$, $|E|>0$. Consider the solution $c(x,0)$ of $-y''+q(x)y=0$. Note that $c(x,0)$ and its derivative $c'(x,0)$ satisfy
\be\label{eq:3.40}
c(x,0)=1+\int_0^x(x-t)q(t)c(t,0)dt,\quad c'(x,0)=\int_0^x q(t)c(t,0)\, dt.
\ee
Since $q\ge 0$ on $\R_+$, it follows from \eqref{eq:3.40} that $c(.,0)$ is positive and nondecreasing on $\R_+$. 

Let us show that there are $C>0$ and $x_0>0$ such that $c(x,0)\ge Cx$ for all $x>x_0$. Firstly, notice that it suffices to prove this claim for compactly supported potentials. Indeed, if $q\ge \tilde{q}$ on $\R_+$, then $c(x,0)\ge \tilde{c}(x,0)$ on $\R_+$.
So, assume that $q$ has a finite support. Then equation $-y''+q(x)y=0$ has two linearly independent solutions $y_1, y_2$ such that
\[
y_1(x)= 1,\ \quad y_2(x)= x,\ \quad x>b.
\]
However, \eqref{eq:3.40} implies  
\[
c'(x,0)=\int_0^b q(t)c(t,0)\,dt=C,\quad x>b,
\]
and hence $c(x,0)=Cx+C_0$ if $x>b$. 

The latter immediately implies that $\frac{1}{c(.,0)}\in L^2(\R_+)$ if $q\neq 0$ and hence, by Theorem \ref{th:help_q}(ii), the inequality \eqref{eq:help_wr=1} is not valid. 
\end{proof}

\begin{remark}
(i) Firstly, let us mention that Corollary \ref{cor:3.11} in not true for nonconstant $w$ and $r$.

(ii) Using the asymptotic of the $m$-function at $0$ obtained in \cite[Lemma 4.1]{KKM_09} (see also formula (4.16) in \cite{KKM_09}), we can deduce from Theorem \ref{th:ev71}: {\em if $q\in L^1_{\loc}(\R_+)$ (not necessarily nonnegative) satisfies}
\be
\int_{\R_+}(1+x)|q(x)|dx<\infty,
\ee
 {\em then the HELP inequality \eqref{eq:help_wr=1} is valid precisely if either  $c(.,0)$ or $s(.,0)$ is bounded on $\R_+$}. 
%Note that $q$ might be negative.
\end{remark}

%%%%%%%%%%%%%%%%%%%%%%%%%%%%
\section{The similarity problem for $J$-nonnegative operators}\label{sec:sim}
  %%%%%%%%%%%%%%%%%%%%%%%%%%%%%%%%%%%%
    %%%%%%%%%%%%%%%%%%%%%%%%%%%%%%%%%%%%
    The main objective of this section is the similarity of the operator
  \be\label{eq:de}
A=\frac{(\sgn\, x)}{w(x)}\big(-(\frac{1}{r(x)}f')'+q(x)f\big),
\ee  
   acting in the Hilbert space $L^2_w(-b,b)$ to a self-adjoint operator. Namely, the operator $A$ (see below for the precise definition) is non-self-adjoint in $L^2_w(\cI)$. Moreover, it is a rank $2$ non-self-adjoint extension of a symmetric operator $A_{\min}$ (see \eqref{eq:LA} below). Under some additional assumptions the spectrum of $A$ is real (cf., e.g., \cite{KM_08, Kar2, KarKos}).  
   %Motivated by certain problems in mathematical physics and transport theory (cf. \cite{Kar1}), we are interested in the problem  %whether $A$ is similar to a self-adjoint operator. 
   The central result of this section is the similarity criterion in the case of even coefficients $w, r, q$. Moreover, we shall show that this problem is closely connected with the HELP inequality \eqref{eq:help}.

\subsection{Differential operators}\label{ss:2.01}
  %%%%%%%%%%%%%%%%%%%%%%%%%%%%%%%%%%%%
  Consider the following differential expressions
\be\label{eq:de}
\mathfrak{a}[f]:=\frac{(\sgn\, x)}{w(x)}\big(-(\frac{1}{r(x)}f')'+q(x)f\big),\qquad \ell[f]:=\frac{1}{w(x)}\big(-(\frac{1}{r(x)}f')'+q(x)f\big).
\ee
\begin{hypot}\label{hyp:01}
 $q\in L^1_{\loc}(\cI)$ is real and $r,w\in L^1_{\loc}(\cI)$ are positive a.e. on $\cI$. % In $L^2((-b,b);wdx)$.
\end{hypot}
Assuming that the coefficients satisfy Hypothesis \ref{hyp:01}, one associates with \eqref{eq:de} the following operators
\be\label{eq:oper}
Af=\ga[f],\quad f\in \dom(A);\qquad Lf=\ell[f],\quad f\in \dom(L),
\ee
where
\be\label{eq:dom_a}
\dom(A)=\dom(L)=\{f\in L^2_w(\cI):\ f,\frac{1}{r}f'\in AC_{\loc}(\cI),\  \ell[f]\in L^2_w(\cI)\},
\ee
%\be\label{eq:dom_l}
%\dom(L)=\{f\in L^2(-1,1):\ f,|r|^{-1}f'\in AC[-1,1], \ (|r|^{-1}f')(\pm 1)=0,\ \ell[f]\in L^2\}.
%\ee
Moreover, we assume the following hypothesis:
\begin{hypot}\label{hyp:02}
The operator $L$ associated with $\ell$ in $L^2_w(\cI)$ is nonnegative and self-adjoint, $L=L^*\ge 0$.
\end{hypot}  

%Let us mention that our approach allows to treat a general case and we need Hypothesis \ref{hyp:02} in order to simplify %considerations. 
 
Consider also the minimal and maximal domains
\be\label{eq:dom_min}
\mathfrak{D}_{\min}=\{f\in \dom(L):\ f(0)=(\frac{1}{r}f')(0)=0\},
\ee
and
\be\label{eq:dom_max}
\mathfrak{D}_{\max}=\{f\in L^2_w(\cI):\ f, \frac{1}{r}f'\in AC_{\loc}(\cI\setminus\{0\}),\  \ell[f]\in L^2\}.
\ee
Define the operators
\be\label{eq:LA}
L_{\min}f=\ell[f], \quad A_{\min}f=\ga[f],\quad \dom(L_{\min})=\dom(A_{\min})=\mathfrak{D}_{\min},
\ee
and
\[
A_{\max}f=\ga[f], \quad L_{\max}f=\ell[f],\quad \dom(L_{\max})=\dom(A_{\max})=\mathfrak{D}_{\max}.
\]
Note that the operators $L_{\min}$ and $A_{\min}$ are symmetric, $n_\pm(L_{\min})=n_\pm(A_{\min})=2$, and
\[
L_{\min}^*=L_{\max},\qquad A_{\min}^*=A_{\max}.
\]
Moreover, 
\[
A=JL,\quad A_{\min}=JL_{\min},\quad A_{\max}=JL_{\max},
\]
where $J:f(x)\to (\sgn\, x) f(x)$. Note that $J=J^*=J^{-1}$ in $L^2_w(\cI)$.

%\subsection{Weyl--Titchmarsh $m$-functions}\label{ss:WF}
%Let $c(x,\lambda)$ and $s(x,\lambda)$ be the solutions of $\ell[y]=\lambda y$ satisfying the initial conditions
%\be\label{eq:fs}
%c(0,\lambda)=(|r|^{-1}s')(0,\lambda)=1,\quad s(0,\lambda)=(|r|^{-1}c')(0,\lambda)=0.
%\ee
%Define 
As in Section \ref{ss:a2}, let $m_+$ and $m_-$ be the $m$-functions associated with the differential expression $\ell$ on $\cI_+=(0,b)$ and $\cI_-=(-b,0)$, respectively. Namely, %the Weyl solutions corresponding to the Neumann boundary conditions at $x=\pm 1$
\be\label{eq:weyl_s}
 \psi_\pm(x,\lambda)=s(x,\lambda)\mp m_\pm(\lambda)c(x,\lambda),\quad \psi_\pm\in L^2_w(\cI_\pm),\qquad (\lambda\in\C_+).
\ee

Further, note that the deficiency subspaces of $L_{\min}$ and $A_{\min}$ are given by
\be\label{eq:defect_L}
           \cN_\lambda(L_{\min})=\Span\{\psi_+(x,\lambda)\chi_+(x),\psi_-(x,\lambda)\chi_-(x)\}
\ee
 and
\be\label{eq:defect_A}
\cN_\lambda(A_{\min})=\Span\{\psi_+(x,\lambda)\chi_+(x),\psi_-(x,-\lambda)\chi_-(x)\},\quad (\lambda\in\C\setminus\R),
\ee
and by the von Neumann formula the maximal domain admits the representation
\be\label{eq:neum}
\gD_{\max}=\gD_{\min}+\cN_{\lambda}+\cN_{\overline{\lambda}},\quad (\lambda\in\C_+).
\ee

Note that the nonnegativity assumption in Hypothesis \ref{hyp:02} can be explicitly characterized in terms of $m$-functions $m_+$ and $m_-$. 

\begin{lemma}[\cite{KM_08}]\label{lem:4.3}
Let the operator $L$ be given by \eqref{eq:oper}, \eqref{eq:dom_a}. Assume also that $L=L^*$ and $m_+$, $m_-$ be the $m$-functions defined by \eqref{eq:weyl_s}. Then the operator $L$ is nonnegative in $L^2_w(\cI)$ if and only if 

\be\label{eq:positive}
-\frac{1}{m_+} - \frac{1}{m_-}\in (S^{-1}).
\ee

In particular, if $w,r,q$ are even, then $L$ is nonnegative if and only if $m_+\in (S)$.
\end{lemma}

Here $(S)$ and $(S^{-1})$ are the Krein--Stieltjes classes (for definitions and properties see \cite{KK1}).

Consider the following extension $A_c$ of the operator $A_{\min}$ 
\begin{eqnarray} \label{eq:Ac}
%\A_{\Dg,\Cg}:=A_{\min}^*\upharpoonright \dom(\A_{\Cg,\Dg}),\quad\\
\dom(A_c)= \Big\{f\in \dom(A_{\min}^*) : \
\begin{array}{c}f(+0)=f(-0)\\
(\frac{1}{r}f')(+0)=   c (\frac{1}{r}f')(-0) 
\end{array}\Big\}.
\end{eqnarray}
%\begin{rem}
Note that the operator $A$ defined by
\eqref{eq:oper} coincides with $A_1$.
%The operator $A_{1,\Cg}$ belongs to the class of operators with
%the so-called $\delta'$-interaction at zero
%\end{rem}

We need the following result (see \cite{KarKos} and \cite[Proposition 3.3]{KKM_09}).

\begin{proposition}[\cite{KKM_09}] \label{prop:res}
\begin{itemize}
%\item [(i)] The operators $A_{1} $ and $A$ coincide, i.e., $A_{1}=A$.
%
\item[ (i)] $A_c=A_c^*$ if and
only if $c=-1$.
\item [(ii)] 
\[ 
\sigma (A_c) \setminus \R = \{z \in \C_+ \cup \C_- :   c m_+ (z) +m_- (-z)= 0\}.
\]
\item [(iii)] If $ z \in
\rho(A_c)\setminus\R $, then for all $f\in L^2_w(\cI)$,
\begin{equation}\label{e III_01}
(A_c-z )^{-1}f =(A_0 - z )^{-1} f + \
\frac{\mathcal{F}_+ (f,z)-\cF_-(f, z)} { c m_+ (z) + m_-
(-z) } \ \left( \, c \psi_+(\cdot,
z) + \psi_-(\cdot, z) \, \right),
\end{equation}
where
%$\mathcal{F}_\pm(f, \lambda)$ denotes the generalized
%Fourier transforms of $f$,\\
\be\label{eq:f_pm}
\mathcal{F}_\pm(f, z):=
\int_{\cI_\pm}f(x)\psi_\pm(x, \pm z)\,w(x) dx.
\ee
\item[(iv)] If Hypothesis \ref{hyp:02} is satisfied, then the spectrum of $A=A_{1}$ is real, $\sigma(A)\subseteq\R$. 
\end{itemize}
\end{proposition}

\subsection{The similarity criteria}

In this section we present several criteria for the similarity of the operator $L$ to a self-adjoint operator in the case of even coefficients $w,r,q$. 

\subsubsection{The similarity criterion in terms of $m$-functions}
Note that the operator $A$ is $J$-self-adjoint and $J$-nonnegative in $L^2_w(\cI)$ if Hypotheses \ref{hyp:01} and \ref{hyp:02} are satisfied. Moreover, the spectrum $\sigma(A)$ of $A$ is real and hence $A$ admits a spectral function $E_A(\cdot)$ (for further details we refer to \cite{Lan82} and also \cite{KarKos}, \cite{KKM_09}). The spectral function (or the family of spectral projections) of $A$ might be unbounded only at $0$ and at $\infty$. In this case, the corresponding point is called {\em a singular critical point}. Critical points, which are not singular, are called {\em regular}.

In this subsection we present several criteria for the similarity of the operator $A$ with even coefficients to a self-adjoint operator. We begin with the following result.

\begin{theorem}\label{th:simcr_m}
Let the operator $A$ be given by \eqref{eq:oper}, \eqref{eq:dom_a}. Assume that Hypotheses \ref{hyp:01} and \ref{hyp:02} are satisfied. Assume additionally that the coefficients $q,r,w$ are even functions. Then: 
\begin{itemize}
\item[(i)] the critical point $\infty$ of $A$ is regular if and only if
\be\label{eq:crit_mB}
\sup_{y>1}\frac{\im m_+(\I y)}{\re m_+(\I y)}=C_\infty<\infty.
\ee 
\item[(ii)] if additionally $\ker(A)=\ker(A^2)$, then the critical point $0$ of $A$ is regular if and only if
\be\label{eq:crit_mC}
\sup_{y\in(0,1)}\frac{\im m_+(\I y)}{\re m_+(\I y)}=C_0<\infty.
\ee
\item[(iii)] the operator $A$ is similar to a self-adjoint operator if and only if
\be\label{eq:crit_m}
\sup_{y>0}\frac{\im m_+(\I y)}{\re m_+(\I y)}=C<\infty.
\ee  
\end{itemize}
\end{theorem}

Notice that the necessity of conditions \eqref{eq:crit_mB}--\eqref{eq:crit_m} was established in \cite{KarKos}. The proof of sufficiency is based on the Veseli\'c--Akopyan criterion (see Theorem \ref{th:veselic} below) and will be given in Section \ref{ss:proof}.

\begin{remark}
Let us mention that $\infty$ is always a critical point of the operator $A$. In the case of even coefficients $w,r,q$, the point $0$ is critical for the operator $A$ if and only if $0\in\sigma_{\ess}(A)$. In particular, in the case $q=\bold{0}$, $0$ is a critical point for the operator $A$ if $w,r\notin L^1(0,b)$. Namely, notice that $0\in\sigma_{\ess}(A)$ precisely if $0\in \sigma_{\ess}(L)$, where $L=JA$. However, if $w,r \notin L^1(0,b)$, then by \cite[Theorem 3]{KK58}, $0\in \sigma(L)$ but $0\notin\sigma_p(L)$. 
\end{remark}
% Using Theorem \ref{th:simcr_m} we obtain the following regularity criteria for critical points of $A$.

\subsubsection{The case $q=\bold{0}$}
Using the results on asymptotic behavior of $m$-functions from Subsections \ref{ss:a01}--\ref{ss:a02}, we obtain the following similarity criterion for the operator 
\be\label{eq:a_q=0}
A=-{(\sgn\, x)}\frac{d}{w(x)dx}\frac{d}{r(x)dx}
\ee
in terms of coefficients $w,r$.

\begin{theorem}\label{th:sim_wr}
Assume Hypotheses \ref{hyp:01} and \ref{hyp:02}. Assume additionally that $q=\bold{0}$. Let $W,R$ be the corresponding distribution functions, $W=\int_0^x wdt$, $R=\int_0^x rdt$. Then:
\begin{itemize}
\item[(i)] If $r\in L^1(0,b)$, then the operator \eqref{eq:a_q=0} is similar to a self-adjoint operator if and only if $W\circ R^{-1}$ is positively increasing at $0$.
\item[(ii)] If $w\in L^1(0,b)$ and $r\notin L^1(0,b)$, then the operator \eqref{eq:a_q=0} is not similar to a self-adjoint operator.
\item[(iii)] If $w,r\notin L^1(0,b)$, then the operator \eqref{eq:a_q=0} is similar to a self-adjoint operator if and only if $W\circ R^{-1}$ is positively increasing at both $0$ and $\infty$. Moreover, the critical point $0$ ($\infty$) is regular if and only if $W\circ R^{-1}$ is positively increasing at $\infty$ ($0$).
\end{itemize}
\end{theorem}

\begin{proof}
Combining Theorem \ref{th:simcr_m} with Lemma \ref{lem:b<inf}, we prove (i) and (ii). (iii) follows by combing Theorem \ref{th:simcr_m} with Corollaries \ref{cor:2.1} and \ref{cor:2.2}.
\end{proof}

\begin{corollary}\label{cor:4.1}
Assume that $q=\bold{0}$, $w,r\notin L^1(0,b)$ are positive a.e. and let $W,R$ be the corresponding distribution functions. 
\begin{itemize}
\item[(i)] if $W\circ R^{-1}$  varies slowly at $0$ ($\infty$), then the critical point $\infty$ ($0$) is singular and the operator \eqref{eq:a_q=0} is not similar to a self-adjoint operator.
\item[(ii)] if $W\circ R^{-1}$ is a regularly varying function with index $\alpha>0$ at $0$  ($\infty$), then the critical point $\infty$ ($0$) of the operator $A$ is regular. The operator \eqref{eq:a_q=0} is similar to a self-adjoint operator if $W\circ R^{-1}$ is a regularly varying function with a positive index at both $0$ and $\infty$.
\end{itemize}
\end{corollary}

\begin{proof}
(i) It suffices to notice that a slowly varying function is not positively increasing (see Appendix \ref{ap:osv}) and then to apply Theorem \ref{th:sim_wr}(iii).

(ii) Again, notice that a regularly varying function is positively increasing and then apply Theorem \ref{th:sim_wr}(iii).  
\end{proof}

\begin{corollary}\label{cor:4.2}
Assume that $q=\bold{0}$, $r=\bold{1}$ and $w\in L^1_{\loc}(-b,b)$ is positive a.e. and even.
\begin{itemize}
\item[(i)] If 
\be
w(x)=x^{\alpha_0-1}l_0(x),\quad x\in (0,x_0), %\cI_i,\quad i\in \{0,\infty\},
\ee
where $\alpha_0>0$ and $l_0$ is a slowly varying function at $0$, then $\infty$ is a regular critical point of the operator $A$.
\item[(ii)] If $b=+\infty$ and 
\be\label{eq:4.21}
w(x)=x^{\alpha_\infty-1}l_\infty(x),\quad x\in (x_0,+\infty), %\cI_i,\quad i\in \{0,\infty\},
\ee
where $\alpha_\infty>0$ and $l_\infty$ is a slowly varying function at $\infty$, then $0$ is a regular critical point of the operator $A$.
\item[(iii)] If $b=+\infty$ and conditions (i) and (ii) are satisfied, then $A$
 is similar to a self-adjoint operator.
 \end{itemize}
\end{corollary}

\begin{proof}
(i) Noting that $l_0(1/x)$ is slowly varying at infinity and using \cite[Proposition IV.5.1]{Kor04}, we obtain
\[
R(x)=\int_0^x r(t)dt=\int_0^x t^{\alpha_0-1}l_0(t)dt =-\int_{1/x}^\infty t^{-1-\alpha_0}l_0(1/t)dt\sim \frac{x^{\alpha_0}}{\alpha_0}l_0(x),\quad x\to 0.
\]
Hence $R$ is regularly varying at $0$ with index $\alpha_0>0$ and by Theorem \ref{th:sim_wr}(iii), $\infty$ is a regular critical point.

(ii) Again, by \cite[Proposition IV.5.1]{Kor04}, 
\[
\int_{x_0}^x r(t)dt=\int_{x_0}^xt^{\alpha_\infty-1}l_\infty(t)dt\sim \frac{x^{\alpha_\infty}}{\alpha_\infty}l_\infty(x),\quad x\to \infty,
\]
and hence $R$ is regularly varying at $\infty$ with index $\alpha_\infty>0$. Therefore, by Theorem \ref{th:sim_wr}(iii), $0$ is a regular critical point for $A$.

(iii) Notice that \eqref{eq:4.21} implies $0\notin \sigma_p(A)$. Therefore, by (i) and (ii), $A$ has no singular critical points and hence $A$ is similar to a self-adjoint operator.
\end{proof}

\begin{corollary}\label{cor:4.3}
Assume that $q=\bold{0}$, $r=\bold{1}$ and $w\in L^1_{\loc}(\R)$ is positive a.e. and even. If there is $x_0>0$ such that
\be\label{eq:4.21B}
w(x)=x^{\alpha-1}p(x),\quad x\in (x_0,+\infty), %\cI_i,\quad i\in \{0,\infty\},
\ee
where $\alpha>0$ and $p:(x_0,+\infty)\to (0,+\infty)$ satisfies 
\be\label{eq:fadd}
p(x)=c_0+g(x),\quad \int_{x_0}^x t^{\alpha-1} g(t)\, dt=o(x^{\alpha}),\quad, x\to\infty,
\ee
then $0$ is a regular critical point of the operator $A$.
\end{corollary}

\begin{proof}
Let us show that $W(x)=\int_0^x w\, dt$ varies regularly at $\infty$ with index $\alpha$. Indeed, for $x\ge x_0$ we get
\[
W(x)-W(x_0)=\int_{x_0}^x t^{\alpha-1}(c+g(t))dt=x^{\alpha}\Big(\frac{c}{\alpha}-\frac{x_0^\alpha}{x^{\alpha}}+\frac{1}{x^{\alpha}}\int_{x_0}^x t^{\alpha-1}g(t)dt\Big)\sim \frac{c}{\alpha}x^{\alpha}
\]
as $x\to \infty$. 
\end{proof}

\begin{remark}\label{rem:4.10}
Let us notice that in the case $b=+\infty$, $q=\bold{0}$ and $r=\bold{1}$ the similarity problem for the operator $A$ has been studied by several authors \cite{CN95}, \cite{FSh2}, \cite{FN98} and the strongest result was obtained in \cite{AK06} (see also \cite[\S 7]{KKM_09}). Namely, the similarity was established under the following conditions:
\[
w(x)=x^{\alpha_0-1}p_0(x), \quad x\in(0,a),\quad w(x)=x^{\alpha_\infty-1}p_\infty(x), \quad x\in(b,\infty),
\]
where $\alpha_0,\alpha_\infty>0$, $p_0$ is continuous at $0$ and $p_0(0)>0$ and there is $c>0$ such that $p_\infty$ satisfies
\be\label{eq:4.30}
\int_{b}^\infty x^{\frac{\alpha_\infty-1}{2}}|p_\infty(x)-c|dx<\infty.
\ee
Clearly, the latter is a particular case of  \eqref{eq:fadd}. 
\end{remark}

\subsubsection{The case $q\neq \bold{0}$}
Consider now a more general situation. Let $r=\bold{1}$ on $\cI$ and $w,q$ be even and such that Hypotheses \ref{hyp:01} and \ref{hyp:02} hold true, that is, the operator 
 \be\label{eq:a.q}
 A=\frac{(\sgn\, x)}{w(x)}\Big(-\frac{d^2}{dx^2}+q(x)\Big)
 \ee
 is $J$-self-adjoint and $J$-nonnegative. Then by Lemma \ref{lem:4.3} we conclude that the solutions $c(.,0)$ and $s(.,0)$ of $-y''+q(x)y=0$ are positive on $\R$. Applying the Liouville transformation from Appendix \ref{ap:LT} and using Proposition \ref{prop:LT} and \cite[Theorem 2.6]{Kar2}, we find that $A$ is similar to 
the following operator
\be\label{eq:4.25}
\tilde{A}=-\frac{(\sgn\, \xi)}{\tilde{w}(|\xi|)}\frac{d^2}{d\xi^2}
\ee
acting in $L^2_{\tilde w}(-B,B)$, where
\be\label{eq:4.26}
\tilde{w}(\xi)=w(x)c^4(x,0),\quad \xi=\xi(x)=\int_0^x \frac{dt}{c^2(t,0)},\quad B=\lim_{x\to b}\xi(x).%,\quad x\in(0,b).
\ee
Moreover, 
\be
\tilde{W}(\xi)=\int_0^\xi \tilde{w}(\mu)d\mu=\int_0^x w(t)c^2(t,0)\, dt,\quad x\in(0,b).
\ee

 Now, applying Theorem \ref{th:sim_wr}, we arrive at the following 
 
\begin{lemma}\label{lem:sim_q}
Let $w,q\in L^1_{\loc}(\cI)$ be even and such that Hypotheses \ref{hyp:01} and \ref{hyp:02} hold true. Let also $c(x,0)$ be the solution of $-y''+q(x)y=0$ such that $c(0,0)=1$ and $c'(0,0)=0$. %the functions $w, W, x$ be defined by \eqref{}. 
Then:
\begin{itemize}
\item[(i)] if ${c(.,0)}\in L^2_w(\cI)$, then the operator $A$ given by \eqref{eq:a.q} is not similar to a self-adjoint operator,
\item[(ii)] if $\frac{1}{c(.,0)}\in L^2(\cI)$, then the operator $A$  is similar to a self-adjoint operator precisely if the function $W$ is positively increasing at $0$,
\item[(iii)] if $c(.,0)\notin L^2_w(\cI)$, $\frac{1}{c(.,0)}\notin L^2(\cI)$, then 
 the operator $A$ given by \eqref{eq:a.q} is similar to a self-adjoint operator if and only if the function $\tilde{W}(\xi)$ given by \eqref{eq:4.25}, \eqref{eq:4.26} is positively increasing at $0$ and at infinity.
\end{itemize}
\end{lemma}

\begin{proof}
(i) Since $\tilde{w}\in L^1(0,B)$ in this case, Theorem \ref{th:sim_wr}(ii) proves the claim.

(ii) In this case we get $B<\infty$ and hence by Theorem \ref{th:sim_wr}(i), $A$ is similar to a self-adjoint operator precisely if the function $\tilde{W}$ is positively increasing at $0$. However, since $c(x,0)\sim 1$ as $x\to 0$, we conclude that $\tilde{W}(\xi)\sim W(x)$ and $\xi\sim x$ as $x\to0$.

(iii) Follows from Theorem \ref{th:sim_wr}(iii).
\end{proof}

In the case $w=\bold{1}$, we immediately obtain the following

\begin{corollary}\label{cor:4.12}
Assume that $b=+\infty$ and $w=\bold{1}$. Let also the assumptions of Lemma \ref{lem:sim_q} be satisfied.  Then:
\begin{itemize}
\item[(i)] if ${c(.,0)}\in L^2(\R_+)$, then the operator $A$ given by \eqref{eq:a.q} is not similar to a self-adjoint operator,
\item[(ii)] if $\frac{1}{c(.,0)}\in L^2(\R_+)$, then the operator $A$  is similar to a self-adjoint operator,
\item[(iii)] if $c(.,0),\frac{1}{c(.,0)}\notin L^2(\R_+)$, then 
 the operator $A$ given by \eqref{eq:a.q} is similar to a self-adjoint operator if and only if the function $\tilde{W}(\xi)$ given by \eqref{eq:4.25}, \eqref{eq:4.26} is positively increasing at infinity,
 \item[(iv)] if additionally $q\ge 0$ on $\R$, then $A$ is similar to a self-adjoint operator.
\end{itemize}
\end{corollary}

\begin{proof}
(i)--(iii) is immediate form Lemma \ref{lem:sim_q}. 
Moreover, (iv) follows from (ii) since under the positivity assumption we get $1/c(.,0)\in L^2(\R_+)$ (see the proof of Corollary \ref{cor:3.11}).
\end{proof}

The next result was established under an additional assumption in \cite{Pyat12} (see Theorem 3.4 in \cite{Pyat12}).

\begin{corollary}[\cite{Pyat12}]\label{cor:pyat}
Let $\cI=\R$ and let $w,q\in L^1_{\loc}(\R)$ be even and satisfying Hypothesis \ref{hyp:01}. Assume also that $w\notin L^1(\R)$, $q\ge 0$ on $\R$ and, moreover, $q>0$ on a set of positive Lebesgue measure. Then the operator $A$ given by \eqref{eq:a.q} is similar to a self-adjoint operator if and only if  the function $W$ is positively increasing at $0$.
%
%In particular, if $w=x^\alpha$, $\alpha>-1$, and $q\ge 0$ on $\R$, then $A$ is similar to a self-adjoint operator.
\end{corollary}

\begin{proof}
Firstly, we note that Hypothesis \ref{hyp:02} is satisfied. Indeed, since $w\notin L^1(\R)$ and $q\ge 0$, the operator $L$ is self-adjoint and nonnegative as a sum $L=L_0+Q$ of two self-adjoint and nonnegative operators (cf. \cite[Theorem VI.1.31]{Kato} and also \cite[\S VI.4.1]{Kato})
\[
L_0:=-\frac{d^2}{wdx^2},\quad Q:=q/w.
\]

Further, since $q>0$ on a set of a positive measure, we conclude that $1/c(.,0)\in L^2(\R_+)$ (cf. the proof of Corollary \ref{cor:3.11}). Applying Lemma \ref{lem:sim_q}(ii), we prove the claim.
\end{proof}

\begin{remark}
Let us mention that using a different approach, Corollary \ref{cor:pyat} was established in \cite{Pyat12} under the additional assumption
\[
q(x)+w(x)\ge \frac{c}{1+x^2},\quad c>0,\quad (x\in\R).
\] 
However, in \cite{Pyat12} the coefficients $w$ and $r$ are not necessarily even. 
%Let us also mention that the special case $q=\bold{0}$ was considered in \cite{CN95} and \cite
\end{remark}

Let us also consider the following particular situation.

\begin{lemma}\label{cor:4.13}
Assume that $b=+\infty$ and $w=\bold{1}$ and the operator $A$ is given by \eqref{eq:a.q}. Let also the assumptions of Lemma \ref{lem:sim_q} be satisfied. Assume additionally that there exist $x_0>0$ and $l\ge -1/2$ such that
\be
q(x)=\frac{l(l+1)}{x^2}+\tilde{q}(x),\quad x\ge x_0,
\ee
where $\tilde q$ satisfies
\be
\int_{x_0}^\infty x|\tilde{q}(x)|\, dx<\infty.
\ee
\begin{itemize}
\item[(i)] If $l\in[-1/2,1/2)$, then the operator $A$ is similar to a self-adjoint operator.
\item[(ii)] If $l= 1/2$, then the operator $A$ is similar to a self-adjoint operator if and only if the solution $c(.,0)$ is unbounded. 
\item[(iii)] If $l> 1/2$, then the operator $A$ is similar to a self-adjoint operator if and only if $c(.,0)\notin L^2(\R)$. 
\end{itemize}
 \end{lemma}
 
\begin{proof}
By \cite[Theorem X.17.1]{har}, equation $-y''+q(x)y=0$ has two linearly independent solutions $y_1, y_2$ such that
\[
y_1(x)\sim x^{l+1},\qquad 
y_2(x)\sim \begin{cases} x^{-l},& l>-1/2\\
\sqrt{x}\log(x), & l=1/2
\end{cases}
  ,\quad x\to\infty.
\]
Firstly, notice that in the cases $l\in [-1/2,1/2]$ both $y_1$ and $y_2$ are not in $L^2(\R_+)$. 

Consider three cases:

1) If $l=-1/2$, then either $c(x,0)\sim C\sqrt{x}$  or $c(x,0)\sim C\sqrt{x}\log(x)$ as $x\to \infty$.
Therefore, either $\xi(x)\sim C^{-2}\log(x)$ and $\ti{W}(\xi)\sim \frac{1}{2}C^2x^2$ as $x\to \infty$, or  $\xi(x)\sim B-\frac{C^{-2}}{\log(x)}$ and $\ti{W}(\xi)\sim \frac{1}{2}C^2x^2\log^2(x)$ as $x\to \infty$. In the first case, we get $B=+\infty$ and the function $\tilde{W}$ is rapidly varying at $\infty$ and hence is positively increasing at $\infty$. By Corollary \ref{cor:4.12}(iii), $A$ is similar to a self-adjoint operator in this case.

Further, if $c(x,0)\sim \sqrt{x}\log(x)$ as $x\to \infty$, then $B<\infty$ and by Corollary \ref{cor:4.12}(ii), $A$ is similar to a self-adjoint operator.

2) If $l>-1/2$ and $c(x,0)\sim Cx^{l+1}$, then $\xi(x)\sim B-\frac{C^{-2}}{2l+1}x^{-2l-1}$ where $B<\infty$. By Corollary \ref{cor:4.12}, $A$ is similar to a self-adjoint operator in this case.

3) If $l>-1/2$ and $c(x,0)\sim Cx^{-l}$, then $\xi(x)\sim \frac{C^{-2}}{2l+1}x^{2l+1}\to \infty$ as $x\to \infty$. If additionally $l>1/2$, then $c(.,0)\in L^2(\R_+)$ and hence by Corollary \ref{cor:4.12}, $A$ is not similar to a self-adjoint operator. If $l\in (-1/2, 1/2]$, then  $c(.,0)\notin L^2(\R_+)$. Next we get
\[
\xi(x)\sim \frac{C^{-2}}{2l+1}x^{2l+1},\quad 
\tilde{W}(\xi)\sim\begin{cases} \log(x), & l=1/2\\ 
                                                \frac{C^2}{1-2l}x^{1-2l}, & |l|<1/2
                                                \end{cases},\quad x\to\infty.
\]
Therefore, we get
\[
\tilde{W}(\xi)\sim\begin{cases} C_1\log(\xi), & l=1/2\\ 
                                                C_2\xi^{\frac{1-2l}{1+2l}}, & |l|<1/2
                                                \end{cases}.
\]
Hence for $l\in (-1/2,1/2)$ the function $\tilde{W}$ varies regularly with index $\frac{1-2l}{1+2l}>0$ at infinity and hence $\tilde{W}$ is positively increasing at $\infty$. Therefore, by Corollary \ref{cor:4.12}(iii), $A$ is similar to a self-adjoint operator. 

Finally, if $l=-1/2$, then $\tilde{W}$ is slowly varying at $\infty$ and hence $\tilde{W}$ is not positively increasing. By Corollary \ref{cor:4.12}(iii), $A$ is not similar to a self-adjoint operator. 
\end{proof}

\begin{remark}
Note that in the case $l=0$ this result was established in \cite[\S 4]{KKM_09} by using a different approach based on a sufficient similarity condition obtained in \cite{KM_08}. Moreover, it was shown in  \cite[\S 5]{KKM_09} that the operator
\[
A=(\sgn\, x)\big(-\frac{d^2}{dx^2}-\chi_{[0,\pi/4]}(|x|)+2\frac{\chi_{\pi/4,+\infty}(|x|)}{(1+|x|-\frac{\pi}{4})^2}\big)
\]
is $J$-nonnegative in $L^2(\R)$ and is not similar to a self-adjoint operator. Clearly, in this case $l=1$ and, moreover, $c(x,0)=(1+|x|-\pi/4)^{-1}$ if $|x|>\pi/4$, which is in $L^2(\R_+)$. Then by Lemma \ref{cor:4.13} it is not similar to a self-adjoint operator.
\end{remark}

\subsubsection{$J$-positive operators with the singular critical point $0$}\label{ss:0crit}
%%%%%%%%%%%%%%%%%%%%%%%%%%%%%%%
%%%%%%%%%%%%%%%%%%%%%%%%%%%%%%%\begin{corollary}

The problem on existence of $J$-positive Sturm--Liouville operators with singular critical points has a long history. As it was already mentioned, only $0$ and $\infty$ may be singular critical points for $J$-positive operators. The existence of $J$-positive Sturm--Liouville operators with the singular critical point $\infty$ was established in \cite{Vol_96} and explicit examples were constructed in \cite{abpyat} and \cite{F98}. Examples of $J$-nonnegative operators with the singular critical point $0$ were first presented in \cite{KarKos}. However, in all these examples $0\in \sigma_p(A)$, that is, operators in these examples are $J$-nonnegative but not $J$-positive. In \cite{AK11}, it was shown that the operator $A=\sgn (\sin\, x)\frac{d^2}{dx^2}$ acting in $L^2(\R)$ has a singular critical point $0$. Clearly, this operator is $J$-positive, however, the weight function is periodic on $\R$ and hence has an infinite number of sign changes. 

Theorem \ref{th:sim_wr} provides a complete characterization of $J$-nonnegative Sturm--Liouville operators of the form \eqref{eq:a_q=0} with even coefficients having singular critical points. Therefore, we obtain a class of $J$-positive operators with the singular critical points $0$ and $\infty$ (cf. Corollary \ref{cor:4.1}). The main aim of this subsection is to present explicit examples of $J$-positive Sturm--Liouville operators with the singular critical point $0$. 

\begin{example}\label{ex:4.1}
Let $l\in L^1_\loc(\R_+)$ be positive a.e. on $\R$ and a slowly varying at $\infty$ function. Assume additionally that there is $x_0>0$ such that $l(x)\ge C>0$ for a.a. $x>x_0$. Consider in $L^2(\R_+)$ the operator
\be\label{eq:a.l}
A_l:=-(\sgn x)\frac{1+|x|}{l(|x|)}\frac{d^2}{dx^2}. 
\ee
Notice that by Theorem \ref{th:karamata}(ii), the function
\[
W(x)=\int_0^x\frac{l(t)}{1+t}dt
\]
is unbounded and slowly varying at $\infty$. Therefore, the operator $A_l$ is $J$-positive in $L^2(\R,\frac{l(|x|)}{1+|x|}dx)$. Moreover, by Corollary \ref{cor:4.1}(i), $0$ is a singular critical point of $A_l$.

For example, setting $l=\bold{1}$, we get $W(x)=\log(1+x)$ and the operator
\[
A_1=-(\sgn x)(1+|x|)\frac{d^2}{dx^2}
\]
is $J$-positive in $L^2(\R,(1+|x|)dx)$ and $0$ is its singular critical point. Let us also mention that the weight $w(x)=\frac{1}{1+|x|}$ is infinitely differentiable at any point of $\R_+$ and hence the regularity of the critical point $0$ does not depend on smoothness of the weight function $w$. This fact was first noticed in \cite{abpyat}.
\end{example}

Using the connection between positively increasing functions at $0$ and at $\infty$, we can modify \cite[Example 1]{abpyat} in order to get one more example.

\begin{example}\label{eq:4.2}
Define the function $w:\R\to\R_+$ as follows:
\be\label{eq:w2}
w(x)=\begin{cases}
|x|^{-1}, & |x|\in \Omega\\
1, & |x|\notin \Omega
\end{cases}, \quad \Omega=\cup_{n=1}^\infty [(2n)!, (2n+1)!].
\ee
Set $a_n=(2n)!$ and $b_n=(2n+1)!$, $n\in \N$. Then we get $\frac{a_n}{b_n}=(2n+1)^{-1}\to 0$ as $n\to\infty$. Moreover,
\[
W(a_n)=\int_0^{a_n}w(t)dt\ge \int_{b_{n-1}}^{a_n}w(t)dt =(2n)!-(2n-1)!>(2n-1)!,
\]
and 
\[
 W(b_n)-W(a_n)=\int_{a_n}^{b_n}w(t)dt=\log(2n+1).
\]
Therefore, 
\[
\frac{W(b_n)}{W(a_n)}=1+\frac{W(b_n)-W(a_n)}{W(a_n)}<1+\frac{\log(2n+1)}{(2n-1)!}\to 1, \quad n\to\infty.
\]
By Lemma \ref{lem:osv}(v), the function $W(x)=\int_0^xw\, dt$ is not positively increasing at $\infty$ and hence, by Theorem \ref{th:sim_wr}, the corresponding operator
\[
A=-\frac{(\sgn x)}{w(x)}\frac{d^2}{dx^2}
\]
has a singular critical point $0$. Moreover, since $w\notin L^1(\R_+)$, we conclude $0\notin \sigma_p(A)$, i.e., $A$ is $J$-positive.
\end{example}

\subsection{Connection with the LRG condition and the HELP inequality}\label{sec:LRG}

The linear resolvent growth (LRG) condition 
\be\label{lrg}
\|(T-\lambda)^{-1}\|_{\cH}\le \frac{C}{\im \, \lambda},\quad (\lambda\in \C\setminus\R),
\ee
is necessary for the similarity of a closed linear operator $T$ acting in a Hilbert space $\cH$ to a self-adjoint operator. 
It was noticed in \cite[Theorem 7.3]{AK_12} that in the regular case, i.e., $b<\infty$ and $q,w,r\in L^1(-b,b)$ are even, condition \eqref{lrg} 
is necessary and sufficient for the operator $A$ to be similar to a self-adjoint operator. Moreover, in \cite{Vol_96}, the connection between the similarity problem and the HELP inequality was observed. Furthermore, it was noticed in \cite{BF2} that in fact the validity of a certain HELP inequality is equivalent to the Riesz basis property of eigenfunctions. In this subsection we extend these results to the case of a singular end-point $x=b$.

\begin{theorem}\label{th:lrg=sim}
Let the operator $A$ be given by \eqref{eq:oper}, \eqref{eq:dom_a}. Assume that Hypotheses \ref{hyp:01} and \ref{hyp:02} are satisfied. Assume additionally that the coefficients $q,r,w$ are even. Then the following are equivalent:
\begin{itemize}
\item[(i)] the operator $A$ is similar to a self-adjoint operator,
\item[(ii)] the operator $A$ satisfies the linear resolvent growth condition \eqref{lrg},
\item[(iii)] the $m$-function $m_+$ satisfies \eqref{eq:crit_m}.
\end{itemize}
If additionally $q=\bold{0}$, then these conditions are further equivalent to the following one:
\begin{itemize}
\item[(iv)] the HELP inequality
\be\label{eq:help_B}
\Big(\int_0^b\frac{1}{w}|f'|^2dx\Big)^2\le K^2\int_0^b|f|^2\, r\, dx\, \int_0^b\frac{1}{r}\big|\big(\frac{1}{w}f'\big)'\big|^2\, dx,\qquad (f\in\dom(A_+)),
\ee
is valid.
%,
%\item[(v)] the function $W\circ R^{-1}$ is positively increasing at $0$ and $\infty$.
\end{itemize}
\end{theorem}

\begin{proof}
The implication $(i)\Rightarrow (ii)$ is well-known. The implication $(ii)\Rightarrow (iii)$ was noticed in \cite{KarKos}. Finally, $(iii)\Rightarrow (i)$ was established in Theorem \ref{th:simcr_m}(iii).

Assume now that $q=\bold{0}$. Then the equivalence $(iii)\Leftrightarrow (iv)$ immediately follows from Theorem \ref{th:ak_crit+} and Lemma \ref{lem:m=1/m}. %Moreover, by Theorem \ref{th:sim_wr}(iii), $(iii)$ is equivalent to $(v)$.  
 \end{proof}

We complete this section with the following 

\begin{remark}
In the regular case, Theorem \ref{th:lrg=sim} was established in \cite{AK_12}. Moreover, in this case the %equivalence $(i)\Leftrightarrow (v)$ is due to Parfenov \cite{Par03}; 
implication $(i)\Rightarrow (iv)$ was observed by Volkmer \cite{Vol_96} and the converse implication $(iv)\Rightarrow (i)$ was noticed in \cite{BF}.
\end{remark}

\section{Proof of Theorem \ref{th:simcr_m}} \label{ss:proof}

Our proof is based on the following criterion  obtained independently by K.~Veseli\'c \cite{Ves72} and R.~Akopjan \cite{Akop_80}.
\begin{theorem}[\cite{Ves72, Akop_80}]\label{th:veselic}
Let $A$ be a $J$-nonnegative operator in a Hilbert space $\mathcal{H}$ such that $\rho(A)\neq\emptyset$. 
Then:
\begin{itemize}
\item[(i)] the critical point $\infty$ of the operator $A$ is regular if and only if the integral
\be\label{eq:veselic_inf}
\int_{1}^\infty \re \big(J(A-\I y)^{-1}f,f\big)_{\mathcal{H}} dy %dz +\int_{I_-}(L-z)^{-1}dz\Big)\in B(\mathcal{H}),
\ee
 converges for all $f\in \mathcal{H}$,
 \item[(ii)] if $\ker(A)=\ker(A^2)$, then the critical point $0$ of the operator $A$ is regular if and only if the integral
\be\label{eq:veselic_0}
\int_{0}^1 \re \big(J(A-\I y)^{-1}f,f\big)_{\mathcal{H}} dy %dz +\int_{I_-}(L-z)^{-1}dz\Big)\in B(\mathcal{H}),
\ee
 converges for all $f\in \mathcal{H}$,
 \item[(iii)]
the operator $A$ is similar to a self-adjoint operator if and only if the following integral %there exists the weak limit
\be\label{eq:veselic}
\int_{0}^\infty \re \big(J(A-\I y)^{-1}f,f\big)_{\mathcal{H}} dy %dz +\int_{I_-}(L-z)^{-1}dz\Big)\in B(\mathcal{H}),
\ee
is convergent for all $f\in \mathcal{H}$.
%Here $I_+=(1,+\infty)$ and $I_-=(0,1)$. %  ($I_+=\pm[\I \frac{1}{y},\I]$ and $I_-=[-\I,-\I \frac{1}{y}]$).
\end{itemize}
\end{theorem}

Before proving Theorem \ref{th:simcr_m}, we need preparatory lemmas.

\begin{lemma}\label{lem:est3}
Let $m_\pm$ be the $m$-functions and $\cF_\pm$ be defined by \eqref{eq:f_pm}. Then 
\[
\int_0^{+\infty}\Big|\re \Big(\frac{\cF_\pm(f,\I y)\cF_\pm(\bar{f},\I y)}{m_+(\I y)-\overline{m}_-(\I y)}\Big)\Big|dy\le \pi\|f\|^2_{L^2_w(\cI_\pm)}
\]
for all $f\in L^2_w(\cI)$.
\end{lemma}

\begin{proof}
Let $T$ be a self-adjoint operator in a Hilbert space $\cH$. %Assume also that $0\in\rho(T)$. 
Then, using the spectral theorem and Fubini's theorem, we get % $f\in\cH$
\begin{align*}
&\int_{0}^{+\infty}\big|\re\big((T-\I y)^{-1}f,f\big) \big|dy\le \int_0^{+\infty}\int_{\R}\frac{|t|}{t^2+y^2}d(E_T(t)f,f)dy\\
&\le  \int_{\R}\Big(\int_0^{+\infty}\frac{|t|}{t^2+y^2}dy \Big)d(E_T(t)f,f)=\frac{\pi}{2}\int_{\R}d(E_T(t)f,f)=\frac{\pi}{2}\|f\|^2_{\cH}.
\end{align*}
Therefore, setting $f=f_\pm$, where $f_\pm$ has supports in $\cI_\pm$, using \eqref{e III_01} with $c=-1$, and nothing that the operators $A_{-1}$ and $A_0$ are self-adjoint, we get
\begin{align*}
&\int_{\R_+}\Big|\re \Big(\frac{\cF_\pm(f,\I y)\cF_\pm(\bar{f},\I y)}{m_+(\I y)-\overline{m}_-(\I y)}\Big)\Big|dy=\int_{\R_+}\big|\re\big((A_{-1}-\I y)f,f\big)-\re\big((A_0-\I y)^{-1}f,f\big)\big|dy\\
&\le \int_{\R_+}\big|\re\big((A_{-1}-\I y)f,f\big)\big|dy+\int_{\R_+}\big|\re\big((A_0-\I y)^{-1}f,f\big)\big|dy\le \pi\|f\|^2_{L^2_w(\cI_\pm)}. 
\end{align*}
\end{proof}

\begin{corollary}\label{cor:est1}
Let $m_\pm$ be the $m$-functions \eqref{eq:weyl_s} and $\cF_\pm$ be defined by \eqref{eq:f_pm}. Then %there is $C_1>0$ such that
\be\label{eq:5.18}
\int_0^{+\infty}\frac{\big|\im \cF_\pm^2(f,\I y)\big|}{\im m_\pm(\I y)}dy \le 2\pi\|f\|^2_{L^2_w(\cI_\pm)}
\ee
and
\be\label{eq:5.19}
\int_{\R_+}\big|\im \cF_\pm^2(f,\I y)\big|\frac{\im m_+(\I y)+\im m_-(\I y)}{|m_+(\I y)-\overline{m_-(\I y)}|^2}dy \le 2\pi\|f\|^2_{L^2_w(\cI_\pm)}
\ee
for all $f=\bar{f}\in L^2_w(\cI)$.
\end{corollary}

\begin{proof}
Clearly, it suffices to establish the estimates for index "+". 
Fix $m_+$ and set $m_-=m_+$. Noting that $f=\bar{f}$, we get 
\[
\re \Big(\frac{\cF_+(f,\I y)\cF_+(f,\I y)}{m_+(\I y)-\overline{m}_+(\I y)}\Big)=\re \Big(\frac{\cF_+^2(f,\I y)}{2\I \im m_+(\I y)}\Big)=\frac{\im \cF_+^2(f,\I y)}{2\im m_+(\I y)}.
\]
Hence, applying Lemma \ref{lem:est3}, we arrive at the first estimate. 
Noting that
\[
\frac{\im m_+(\I y)+\im m_-(\I y)}{|m_+(\I y)-\overline{m_-(\I y)}|^2}\le \frac{1}{|m_+(\I y)-\overline{m_-(\I y)}|}\le \frac{1}{\im m_\pm(\I y)},\quad (y>0),
\] 
we prove the second inequality.
%
%Clearly, it suffices to establish convergence for index "+". 
%Firstly, note that if $T$ is a self-adjoint operator in a Hilbert space $\mathcal{H}$, then the following weak limit exists
%%Since $\tilde{L}$ and $L_0$ are self-adjoint, we get 
%\[
%w-\lim_{y\to+\infty} \int_{1/y}^{ y} \re\, (T-\I y)^{-1}\, dy=\sgn (T)\in B(\mathcal{H}).
%\] 
%Since $\tilde{L}$ and $L_0$ are self-adjoint, there is the following weak limit
%\[
%w-\lim_{y\to+\infty}\Big[\int_{1/y}^{ y} \re\, (\tilde{L}-\I y)^{-1}dy + \int_{1/y}^{ y} \re\, ({L}-\I y)^{-1}dy \Big]=\sgn(\tilde{L})-%\sgn (L).
%\]
%Since this is true for any $L_0$ and $\tilde{L}$, we can set $m_-=m_+$ and $f_+=f\chi_+$ and $f=\bar{f}$. Therefore, we get
%\[
%\Big|\int_{\R_+}\frac{\im \cF_+^2(f_+,\I y)}{\im m_+(\I y)}dy\Big|\le \|(\sgn(\tilde{L})f_+-\sgn (L)f_+,f_+)\|\le C_1\|f_+\|^2,
%\] 
%where $C_1=\|\sgn(\tilde{L})-\sgn (L)\|$.
\end{proof}

\begin{corollary}\label{cor:est2}
Let $m_+$ be the $m$-functions  \eqref{eq:weyl_s} and $\cF_+$ be defined by \eqref{eq:f_pm}. If \eqref{eq:crit_m} holds true, then 
%there is $C_2>0$ such that
%\be\label{eq:5.20}
%\int_{\R_+}|\im \cF_\pm^2(f,\I y)|\frac{|\im(m_+(\I y)-m_-(\I y))|}{|m_+(\I y)+\overline{m}_-(\I y)|^2}dy\Big|\le 4C_0\|f\|%^2_{L^2(\R_\pm,|r|)}
%\ee
%for all $f=\bar{f}\in L^2(\R,|r|)$.
%
%In particular, if $m_+=m_-$ and \eqref{eq:lrg} is satisfied, then
\be\label{eq:5.20}
\int_{\R_+}\frac{|\im \cF_+^2(f,\I y)|}{\re \, m_+(\I y)}dy\le 2\pi C\|f\|^2_{L^2_w(\cI_+)}
\ee
for all $f=\bar{f}\in L^2_w(\cI)$.
\end{corollary}

\begin{proof}
Immediately follows from \eqref{eq:crit_m} and \eqref{eq:5.18}.
%Firstly, observe that
%\[
%\frac{|\im(m_+(\I y)-m_-(\I y))|}{|m_+(\I y)+\overline{m}_-(\I y)|^2}\le \frac{1}{|m_+(\I y)+\overline{m}_-(\I y)|}\le \frac{C_0}%{\im m_\pm(\I y)},\quad y>0.
%\]
%Then, applying Corollary \ref{cor:est1}, we get
%\begin{align*}
%&\Big|\int_{\R_+}\im \cF_\pm^2(f,\I y)\frac{\im(m_+(\I y)-m_-(\I y))}{|m_+(\I y)+\overline{m}_-(\I y)|^2}dy\Big|\\
%&\int_{\R_+}\Big|\im \cF_\pm^2(f,\I y)\Big|\frac{|\im(m_+(\I y)-m_-(\I y))|}{|m_+(\I y)+\overline{m}_-(\I y)|^2}dy\\
%&\le C_0\int_{\R_+}\frac{\big|\im \cF_\pm^2(f,\I y)\big|}{\im m_\pm(\I y)}dy.
%\end{align*}
%Corollary \ref{cor:est1} completes the proof.% with $C_2=4C_0$.
%\[
%\im \cF_\pm^2(f,\I y)=2\re \cF_\pm(f,\I y)\im \cF_\pm(f,\I y).
%\]
%Moreover, it suffices to prove the claim for $f\in L^2$ such that $\hat{f}_\pm$ is nonnegative. So, assume that $\hat{f}_\pm$ is %nonnegative. Therefore, so are $\re \cF_\pm(f,\I y)$ and $\im \cF_\pm(f,\I y)$ for all $y>0$. Noting that by \eqref{eq:lrg_im}
%
%and applying Lemma \ref{lem:est1} arrive at the desired estimate with $C_2=C_1^2$. %we complete the proof.
\end{proof}

\begin{corollary}\label{cor:est4}
Let $m_\pm$ be the $m$-functions  \eqref{eq:weyl_s} and $\cF_\pm$ be defined by \eqref{eq:f_pm}. Then
%If \eqref{eq:lrg_im} holds true, then there is $C_3>0$ such that
\[
\int_{\R_+}\Big|\re \cF_\pm^2(f,\I y)\Big|\frac{|\re(m_+(\I y)-m_-(\I y))|}{|m_+(\I y)-\overline{m}_-(\I y)|^2}dy\le 3\pi\|f\|^2_{L^2_w(\cI_\pm)}
\]
for all $f=\bar{f}\in L^2_w(\cI)$. % such that $\hat{f}_\pm$ has a compact support.
\end{corollary}

\begin{proof}
It suffices to note that
\begin{align*}
&\re \cF_\pm^2(f,\I y)\frac{\re (m_+(\I y)-\overline{m_-(\I y)})}{|m_+(\I y)-\overline{m_-(\I y)}|^2}\\
&=\re \Big(\frac{\cF_\pm^2(f,\I y)}{m_+(\I y)-\overline{m_-(\I y)}}\Big)+
\im \cF_\pm^2(f,\I y)\frac{\im (m_+(\I y)-\overline{m_-(\I y)})}{|m_+(\I y)-\overline{m_-(\I y)}|^2}.
\end{align*}
Applying Lemma \ref{lem:est3} and Corollary \ref{cor:est1}, we complete the proof.
\end{proof}

\begin{corollary}\label{cor:est5}
Let $m_\pm$ be the $m$-functions  \eqref{eq:weyl_s} and $\cF_\pm$ be defined by \eqref{eq:f_pm}. If \eqref{eq:crit_m} holds true, then 
%there is $C_2>0$ such that
%\[
%\int_{\R_+}|\re \cF_\pm^2(f,\I y)|\frac{\re(m_+(\I y)+m_-(\I y))}{|m_+(\I y)+\overline{m}_-(\I y)|^2}dy\le C_2\|f\|^2_{L^2(\R_%\pm,|r|)}
%\]
%for all $f\in L^2(\R,|r|)$. % such that $\hat{f}_\pm$ has a compact support.

%In particular, if $m_+=m_-$ and \eqref{eq:lrg} is satisfied, then
\be\label{eq:5.21}
\int_0^{+\infty}\frac{|\re \cF_+^2(f,\I y)|}{\re m_+(\I y)}dy\le 3\pi(1+9C^2)\|f\|^2_{L^2_w(\cI_+)}
\ee
for all $f=\bar{f}\in L^2_w(\cI)$.
\end{corollary}

\begin{proof}
Applying Corollary \ref{cor:est4} with $m_-=2m_+$, we get
\begin{align*}
&\int_{\R_+}\big|\re \cF_+^2(f,\I y)\big|\frac{|\re(m_+(\I y)-2\overline{m_+(\I y)})|}{|m_+(\I y)-2\overline{m_+(\I y)}|^2}dy\\
&=\int_{\R_+}\big|\re \cF_+^2(f,\I y)\big|\frac{|\re m_+(\I y)|}{|\re m_+(\I y)-3\I \im m_+(\I y)|^2}dy\le 3\pi\|f\|^2_{L^2_w(\R_+)}.
\end{align*}
By \eqref{eq:crit_m}, we get %Noting that
\[
|m_+(\I y)-2\overline{m_+(\I y)}|^2=(\re m_+(\I y))^2+9(\im m_+(\I y))^2\le (1+9C^2)(\re m_+(\I y))^2,
\]
and hence we arrive at the following inequality
\[
\frac{1}{\re m_+(\I y)}\le (1+9C^2)\frac{\re m_+(\I y)}{|\re m_+(\I y)-3\I \im m_+(\I y)|^2},
\]
which completes the proof.
%The latter implies \eqref{eq:5.21}  with $C_2=3\pi(16+4C_0^2)$.% implies 
%Observe that 
%\[
%\re \cF_\pm^2(f,\I y)=(\re \cF_\pm(f,\I y) -\im \cF_\pm(f,\I y))(\re \cF_\pm(f,\I y) +\im \cF_\pm(f,\I y)).
%\]
%Hence
%\[
%\re \cF_\pm^2(f,\I y)=\int_{\R_+}\frac{s-y}{s^2+y^2}\hat{f}_\pm(s)d\tau_\pm(s)\cdot \int_{\R_+}\frac{s+y}{s^2+y^2}\hat{f}_%\pm(s)d\tau_\pm(s).
%\]
%Assume now that $\hat{f}_\pm$ is nonnegative and has a compact support. Then for $\re \cF_\pm^2(f,\I y) \le 0$ for  $y>\sup\%{s:s\in\supp \hat{f}_\pm\}$.  
\end{proof}

\begin{proof}[Proof of Theorem \ref{th:simcr_m}]
We shall prove only part (iii) since the remaining parts can be established similarly.

Firstly, observe that
\[
m_+(z)=m_-(z),\quad \psi_+(x,z)=\psi_-(-x,z)
\]
since $q,w,r$ are even. Therefore,
\begin{align*}
\cF_-(f,z)&=\int_{-b}^0 f(x)\psi_-(x,-z)dx=\int_{-b}^0 f(x)\psi_+(-x,-z)dx\\
&=\int_0^b f(-x)\psi_+(x,-z)dx=\cF_+(f_-,-z),
\end{align*}
where $f_-(x)=f(-x)$.

By Theorem \ref{th:veselic}, we need to show that the integral
\[
\int_0^{+\infty}\re \big(J(A-\I y)^{-1}f,f\big)_{L^2_w(\cI)}dy
\]
converges for all $f\in L^2_w(\cI)$. Using \eqref{e III_01} with $c=1$, we get % implies  
\begin{equation}\label{eq:7.06}
\big((A-\I y )^{-1}f -(A_0 - \I y )^{-1} f, Jf\big) =
\frac{(\mathcal{F}_+(f,\I y)-\mathcal{F}_+(f_-, -\I y))(\mathcal{F}_+(\overline{f},\I y)-\mathcal{F}_+(\overline{f}_-, -\I y))}
{2\re m_+(\I y)}.
\end{equation}
Denote the right hand side in \eqref{eq:7.06} by $\cA(f,\I y)$. 
Snce $A_0$ is self-adjoint, it suffices to show that the following integral
\[
\int_0^{+\infty}\re \cA(f,\I  y)dy
\]
converges for all $f\in L^2_w(\cI)$. 

Denote $f=f^R+\I f^I$, where $f^R=\overline{f^R}$ and $f^I=\overline{f^I}$. Then 
\[
\cF_+(f,\I y)\cF_+(\bar{f},\I y)=\cF_+^2(f^R,\I y)+\cF_+^2(f^I,\I y),
\]
and 
\[
\cF_+(f,\I y)\cF_+(\bar{g},-\I y)+\cF_+(\bar{f},\I y)\cF_+(g,-\I y)=2(\cF_+(f^R,\I y)\cF_+(g^R,-\I y)+\cF_+(f^I,\I y)\cF_+(g^I,-\I y)).
\]
Therefore,
\[
\cA(f,\I y)=\cA(f_R,\I y) + \cA(f_I,\I y),
\]
and hence we can restrict our considerations to the case of real valued $f\in L^2_w(\cI_+)$.

Finally, let $f=\bar{f}\in L^2_w(\cI)$ and denote $f=f_+\chi_{+}(x)+f_-\chi_{-}(x)$. Hence
\begin{align*}
&\frac{\re(\cF_+(f_+,\I y)-\cF_+(f_-,-\I y))^2}{ \re m_+(\I y)}\\
%&+\frac{(F_+(f_+,-\I y)-F_+(f_-,\I y))^2}{2 \re\, m_+(\I y)}\\
%&=\frac{\re\, (F_+^2(f_+,\I y))}{\re\, m_+(\I y)}+\frac{\re\, (F_+^2(f_-,-\I y))}{\re\, m_+(\I y)}-2\frac{\re( F_+(f_+,\I y)F_+(f_-,-%\I y))}{\re\, m_+(\I y)}\\
%&=\frac{\re\, (F_+^2(f_+,\I y))}{\re\, m_+(\I y)}+\frac{\re\, (\overline{F_+^2(f_-,\I y)})}{\re\, m_+(\I y)}-2\frac{\re( F_+(f_+,\I %y)\overline{F_+(f_-,\I y)})}{\re\, m_+(\I y)}\\
&=\frac{1}{\re m_+(\I y)}\Big(\big(\re (\cF_+(f_+,\I y)-\cF_+(f_-,-\I y))\big)^2 - \big(\im (\cF_+(f_+,\I y)-\cF_+(f_-,-\I y))\big)^2\Big)\\
&=\frac{1}{\re m_+(\I y)}\Big((\re \cF_+(f_+-f_-,\I y))^2 - (\im \cF_+(f_+ + f_-,\I y))^2\Big),
\end{align*}

Setting $f_+=f_-$ and then $f_+=-f_-$, we observe that it suffices to show that the following integrals
\[
\int_0^\infty \frac{|\re \cF_+(f,\I y)|^2}{\re m_+(\I y)}dy,\qquad \int_0^\infty \frac{|\im  \cF_+(f,\I y)|^2}{\re m_+(\I y)}dy,
\]
converge for all $f=\bar{f}\in L^2_w(\cI)$, or equivalently,
\[
\int_0^\infty \frac{|\cF_+(f,\I y)|^2}{\re m_+(\I y)}dy<\infty,\quad (f=\bar{f}\in L^2_w(\cI)).
\]
However,
\[
\int_0^\infty \frac{|\cF_+(f,\I y)|^2}{\re m_+(\I y)}dy=\int_0^\infty \frac{|\cF_+^2(f,\I y)|}{\re m_+(\I y)}dy\le \int_0^\infty \frac{|\re \cF_+^2(f,\I y)|+|\im \cF_+^2(f,\I y)|}{\re m_+(\I y)}dy.
\]
Applying Corollaries \ref{cor:est2} and \ref{cor:est5}, we complete the proof.
\end{proof}

%%%%%%%%%%%%%%%%%%%%%%%%%%%%%%%%%%%
%%%%%%%%%%%%%%%%%%%%%%%%%%%%%%%%%%%

\section{On the well-posedness for the stationary Fokker--Plank equation}\label{sec:FP}

Consider the simplest two--way diffusion equation
\be\label{eq:FP}
(\sgn x)w(x)u_t(x,t)=(\frac{1}{r(x)}u_{x}(x,t))_{x}-q(x)u(x,t),\quad (0<t<t_0\le +\infty,\quad x\in\cI).
\ee
Here $q\in L^1_{\loc}(\cI)$ and $r,w\in L^1_{\loc}(\cI)$ satisfy $w, r>0$ a.e. on $\cI$.
It is assumed that the function $u$ satisfies 
\be\label{eq:bc}
u(x,0)=\phi_+(x),\ \ \ x\in\cI_+;\quad u(x,t_0)=\phi_-(x),\ \ \ x\in\cI_-.
\ee
If $t_0=\infty$, we should change \eqref{eq:bc} as follows
\be\label{eq:bc_inf}
u(x,0)=\phi_+(x),\ \ \ x\in\cI_+;\quad \int_{\cI}|u(x,t)|^2|p(x)|dx=O(1),\ \ \text{as}\ \  t\to\infty.
\ee
 Moreover, if necessary additional self-adjoint $t$-independent boundary conditions at $x=-b$ and $x=b$ are assumed.

Boundary value problems \eqref{eq:FP}, \eqref{eq:bc} and \eqref{eq:FP}, \eqref{eq:bc_inf} are of the {\em forward--backward} type. They arise in kinetic theory and in the theory of stochastic processes and have a long history. For example, if $\cI=(0,\pi)$, $r(x)=(\sin\, x)^{-1}$, $p(x)=\cos x\sin x$, $q=\bold{0}$, then equation \eqref{eq:FP}, derived by Bothe in 1929 \cite{bothe},  describes the steady-state distribution of particles scattered by a slab. Existence and uniqueness of solutions to the corresponding BVPs as well as the representation of solutions by the eigenfunction expansion was proven by Beals (see \cite{Be81}, \cite{fk}, \cite[Chapter X.6]{gmp}). If $p(x)=x$, $r(x)\equiv1$ and $q(x)=\frac{1}{4}(x^2+a^2-2)$, then \eqref{eq:FP} is the one-dimensional linear stationary Fokker--Plank equation (see \cite{bp, wu} and also \cite[Chapter X.5]{gmp}). The case $\cI=\R$, $q=\bold{0}$, $r=\bold{1}$ and $p(x)=(\sgn x)|x|^\alpha$, $\alpha>-1$, arises in the theory of stochastic processes (see \cite{gmp}, \cite{pag1}), however, existence and uniqueness of solutions to \eqref{eq:FP}, \eqref{eq:bc} was established in \cite[\S 5.1]{Kar1} (note that in \cite{bg}, \cite{pag1}, \cite{pag2} well-posedness of BVPs was studied under additional smoothness assumptions on the initial data $\phi_\pm$). For further examples we refer to \cite{Be85}, \cite{gmp}, \cite{Pyat}, \cite{cvdm}.

Let us also mention that equation
\eqref{eq:FP} belongs to the class of second-order equations with nonnegative characteristic
form. Boundary value problems for this class of equations were considered by various
authors (see \cite{kn}, \cite{or} and references therein). But some restrictions imposed in this
theory makes it inapplicable to equation \eqref{eq:FP}.

Separation of variables in \eqref{eq:FP} leads to the indefinite spectral problem
\be\label{eq:SP}
-(\frac{1}{r(x)}y')'+q(x)y=\lambda \, (\sgn\, x)w(x)y,\quad x\in \cI,
\ee
and the well-posedness issue for the above boundary value problems is closely connected with the similarity problem for the corresponding indefinite Sturm--Liouville operator $A=\frac{(\sgn x)}{w}(-\frac{d}{dx}\frac{d}{rdx}+q)$ considered in Section \ref{sec:sim}. For the case when $A$ is $J$-nonnegative and has purely discrete spectrum, the problem \eqref{eq:FP}, \eqref{eq:bc} has been studied in great detail (see \cite{Be81}, \cite{bp}, \cite{fk}, \cite{gmp}, \cite{Pyat}, \cite{cvdm} and references therein). 
In the general case, the following result holds true (see \cite[Theorem 1.1]{Kar1}).

\begin{theorem}[\cite{Kar1}]\label{th:karabash}
Let $w,r, q\in L^1_{\loc}(\cI)$ be such that Hypotheses \ref{hyp:01} and \ref{hyp:02} are satisfied. Assume also that the operator $A$ is $J$-positive ($\ker(A)=\{0\}$) and similar to a self-adjoint operator. Then the problems \eqref{eq:FP}, \eqref{eq:bc} and \eqref{eq:FP}, \eqref{eq:bc_inf} have unique
strong solutions for each pair $\{\phi_+, \phi_-\}$, $\phi_\pm \in L^2_w(\cI_\pm)$. 
\end{theorem} 

Thus applying the results on the similarity for the operator $A$ from Section \ref{sec:sim} we immediately obtain conditions for the existence and uniqueness of solutions to problems  \eqref{eq:FP}, \eqref{eq:bc} and \eqref{eq:FP}, \eqref{eq:bc_inf}. Let us present only a few of them. Namely, for simplicity we restrict to the case $r=\bold{1}$ and $\cI=\R$, i.e., we consider the following equation
\be\label{eq:FP_r=1}
(\sgn\, x)w(x)u_t(x,t)=u_{xx}(x,t)-q(x)u(x,t),\quad (0<t<t_0\le +\infty,\quad x\in\R),
\ee
Then using Lemma \ref{cor:4.13} we arrive at the following 

\begin{theorem}\label{th:FP}
Let $w,q\in L^1_{\loc}(\R)$ be even and such that Hypotheses \ref{hyp:01} and \ref{hyp:02} are satisfied and $0\notin\sigma_p(A)$. 
Let also $c(x)$ be the solution of $-y''+q(x)y=0$ such that $c(0)=1$ and $c'(0)=0$. %the functions $w, W, x$ be defined by \eqref{}. 
Then for each pair $\{\phi_+, \phi_-\}$, $\phi_\pm \in L^2_w(\R_\pm)$, there is a unique
strong solution $u$ to the problems \eqref{eq:FP_r=1}, \eqref{eq:bc} and \eqref{eq:FP_r=1}, \eqref{eq:bc_inf} if at least one of the following conditions is satisfied:
\begin{itemize}
%\item[(i)] if ${c(.,0)}\in L^2_w(\cI)$, then the operator $A$ given by \eqref{eq:a.q} is not similar to a self-adjoint operator,
\item[(i)] $\frac{1}{c(.)}\in L^2(\R)$ and the function $W(x)=\int_0^xw\, dt$ is positively increasing at $0$,
\item[(ii)]  $\frac{1}{c(.)}\notin L^2(\R)$,  the function $W$ is positively increasing at $0$ and the function $\tilde{W}(\xi)$ given by \eqref{eq:4.25}, \eqref{eq:4.26} is positively increasing at infinity.
\end{itemize}
\end{theorem}

\begin{proof}
It suffices to notice that the assumption $0\notin \sigma_p(A)$ implies $c\notin L^2_w(\R)$ and, moreover, since $c(x)\sim x$ at $0$, we get
\[
\xi(x)\sim x,\quad \ti{W}(\xi)\sim W(x),
\]
 as $x\to 0$, which clearly implies that $\ti{W}$ is positively increasing at $0$ precisely if so is $W$. 
Combining  Lemma \ref{cor:4.13} with Theorem \ref{th:karabash} we complete the proof.
\end{proof}

\begin{corollary}\label{cor:6.3}
Let $q=\bold{0}$ and $w\in L^1_{\loc}(\R)$ be even and positive. Assume that $w\notin L^1(\R)$, that is $W$ is unbounded. If the function $W$ is positively increasing at $0$ and at $\infty$, then  for each pair $\{\phi_+, \phi_-\}$, $\phi_\pm \in L^2_w(\R_\pm)$, there is a unique strong solution $u$ to the problems \eqref{eq:FP_r=1}, \eqref{eq:bc} and \eqref{eq:FP_r=1}, \eqref{eq:bc_inf}.

In particular, if there are $\alpha,\beta>-1$ and the functions $p_0, p_1$ such that
\be
w(x)=x^{\alpha}p_0(x),\quad x\in (0,x_0),\qquad w(x)=x^{\beta}p_1(x),\quad x\ge x_1>0,
\ee
where $p_0 \in C[0,x_0]$ is positive and $p_1$ admits the representation 
\be\label{eq:6.7}
p_1(x)=c+g(x),\quad c>0,\quad \int_{x_1}^x t^\beta g(t)\, dt=o(x^{\beta+1}),
\ee
as $x\to \infty$, then the problems \eqref{eq:FP_r=1}, \eqref{eq:bc} and \eqref{eq:FP_r=1}, \eqref{eq:bc_inf} have unique strong solutions for each pair $\{\phi_+, \phi_-\}$, $\phi_\pm \in L^2_w(\R_\pm)$.
\end{corollary}

The proof immediately follows from Theorem \ref{th:sim_wr}(iii) and Corollaries \ref{cor:4.2}(ii) and \ref{cor:4.3}. 

\begin{remark}
Let us mention that the case $w(x)=|x|^\alpha$, $\alpha>-1$, was studied in \cite{pag1}, \cite{pag2}. To the best of our knowledge, the strongest result was obtained in \cite[Theorem 5.5]{Kar1}. However, in \cite{Kar1} it is assumed that at $\infty$ the weight $w$ satisfies condition \eqref{eq:4.30}, which is much stronger than \eqref{eq:6.7}.
\end{remark}

Let us complete this section by considering the following equation

\be\label{eq:FP_w=x}
x\, u_t(x,t)=u_{xx}(x,t)-q(x)u(x,t),\quad (0<t<t_0\le +\infty,\quad x\in\R),
\ee

\begin{lemma}\label{lem:6.5}
Let  $q\in L^1_{\loc}(\R)$ be even and such that Hypotheses \ref{hyp:01} and \ref{hyp:02} hold true. 
Assume additionally that there exist $x_0>0$ and $l\ge -1/2$ such that
\be
q(x)=\frac{l(l+1)}{x^2}+\tilde{q}(x),\quad x\ge x_0,\quad \int_{x_0}^\infty x|\tilde{q}(x)|\, dx<\infty.
\ee
%where $\tilde q$ satisfies
%\be
%\int_{x_0}^\infty x|\tilde{q}(x)|\, dx<\infty.
%\ee
If at least one of the following conditions are satisfied
\begin{itemize}
\item[(i)]  $l\in[-1/2,1)$,
\item[(ii)]  $l\ge 1$ and $c(x)\notin L^2(\R)$, 
%\item[(iii)]  $l= 1$ and , 
\end{itemize}
then  the problems \eqref{eq:FP_w=x}, \eqref{eq:bc} and \eqref{eq:FP_w=x},  \eqref{eq:bc_inf} have unique strong solutions for each pair $\{\phi_+, \phi_-\}$, $\phi_\pm \in L^2_{|x|}(\R_\pm)$.
\end{lemma}

\begin{proof}
By Theorem \ref{th:karabash}, it suffices to show that the corresponding operator $A=\frac{1}{x}(-d^2/dx^2+q(x))$ is similar to a self-adjoint operator in $L^2_{|x|}(\R)$. 
As in the proof of Lemma \ref{cor:4.13} notice that 
by \cite[Theorem X.17.1]{har}, equation $-y''+q(x)y=0$ has two linearly independent solutions $y_1, y_2$ such that
\[
y_1(x)\sim x^{l+1},\qquad 
y_2(x)\sim \begin{cases} x^{-l},& l>-1/2\\
\sqrt{x}\log(x),& l=-1/2
\end{cases}
  ,\quad x\to\infty.
\]

Firstly, notice that  both $y_1$ and $y_2$ are not in $L^2(\R_+,(1+x)dx)$ precisely if $l\in [-1/2,1]$. 
Consider three cases:

1) If $l=-1/2$, then either $c(x)\sim C\sqrt{x}$  or $c(x)\sim C\sqrt{x}\log(x)$ as $x\to \infty$.
Therefore, by Theorem \ref{th:karamata}, either $\xi(x)\sim C^{-2}\log(x)$ and $\ti{W}(\xi)\sim \frac{C^2}{3}x^3$ as $x\to \infty$, or  $\xi(x)\sim B-\frac{1}{C^2\log(x)}$ and $\ti{W}(\xi)\sim \frac{1}{3}x^3\log^2(x)$ as $x\to \infty$. In the first case, we get $B=+\infty$ and the function $\tilde{W}$ is reapidly varying at $\infty$ and hence is positively increasing at $\infty$. By Corollary \ref{cor:4.12}(iii), $A$ is similar to a self-adjoint operator in this case.

Further, if $c(x)\sim C\sqrt{x}\log(x)$ as $x\to \infty$, then $B<\infty$ and by Corollary \ref{cor:4.12}(ii), $A$ is similar to a self-adjoint operator. 

2) If $l>-1/2$ and $c(x)\sim Cx^{l+1}$  as $x\to \infty$, then $\xi(x)\sim B-\frac{C^{-2}}{2l+1}x^{-2l-1}$ as $x\to \infty$, where $B<\infty$. By Corollary \ref{cor:4.12}, $A$ is similar to a self-adjoint operator in this case.

3) Finally, let $l>-1/2$ and $c(x)\sim Cx^{-l}$  as $x\to \infty$. Then $c\in L^2(\R)$, which is impossible by the assumption.
%
% $\xi(x)\sim \frac{C^{-2}}{2l+1}x^{2l+1}\to \infty$ as $x\to \infty$. If  $l>1$, then $c(.)\in L^2_{x}(\R_+)$ and hence by %Corollary \ref{cor:4.12}, $A$ is not similar to a self-adjoint operator. 
%
%If $l\in (-1/2, 1]$, then  $c(.,0)\notin L^2_{x}(\R_+)$. Next we get
%\[
%\xi(x)\sim \frac{C^{-2}}{2l+1}x^{2l+1},\quad 
%\tilde{W}(\xi)\sim\begin{cases} C^2\log(x), & l=1\\ 
%                                                \frac{C^2}{2-2l}x^{2-2l}, & l\in(-1/2,1)
%                                                \end{cases}.
%\]
%Therefore, we get
%\[
%\tilde{W}(\xi)\sim\begin{cases} C_1\log(\xi), & l=1\\ 
%                                               C_2x^{\frac{2-2l}{1+2l}}, & l\in(-1/2,1)
 %                                               \end{cases}.
%\]
%Hence for $l\in (-1/2,1)$ the function $\tilde{W}$ varies regularly with index $\frac{2-2l}{1+2l}>0$ at infinity and hence is %positively increasing at $\infty$. Therefore, by Corollary \ref{cor:4.12}(iii), $A$ is similar to a self-adjoint operator. 
%
%Finally, if $l=1$, then $\tilde{W}$ is slowly varying at $\infty$ and hence is not positively increasing. By Corollary \ref{cor:4.12}(iii), %$A$ is not similar to a self-adjoint operator. 
\end{proof}
%%%%%%%%%%%%%%%%%%%%%%%%%%%%%%%
%%%%%%%%%%%%%%%%%%%%%%%%%%%%%%%

%%%%%%%%%%%%%%%%%%%%%%%%%%%%%%%%%%%%%%%%%
%%%%%%%%%%%%%%%%%%%%%%%%%%%%%%%%%%%%%%%%%

\appendix

\section{Regularly varying and positively increasing  functions}\label{ap:osv}

Firstly, let us recall the concept of regularly varying functions (see, e.g., \cite{BGT}, \cite[Chapter IV]{Kor04}, \cite{S}).

\begin{definition}\label{def:f_rv}
Let $f:(a,+\infty)\to\R_+$ be measurable and eventually positive.
The function $f$ is called slowly varying at $\infty$ if
\be
\lim_{x\to \infty}\frac{f(xt)}{f(x)}=1,\quad (t>0).
\ee
The function $f$ is called regularly varying  at $\infty$ with index $\alpha\in\R$ if 
\be\label{eq:lo}
\lim_{x\to 0}\frac{f(xt)}{f(x)}=t^\alpha, \quad (t>0).
\ee
If the limit in \eqref{eq:lo} equals $\infty$ for all $t>1$, then $f$ is called rapidly varying at $\infty$ .

The function $f:(0,b)\to \R_+$ is called slowly  (regularly or rapidly) varying at $0$ if the function $\ti{f}(x):=1/f(1/x)$ is 
slowly  (regularly or rapidly) varying at $\infty$.
\end{definition}

%\begin{remark}
%We call $f$ a regularly varying function if it is either slowly or rapidly varying, or regularly varying of index $\alpha$.  
%\end{remark}

Clearly, the class of slowly varying functions coincides with the class of regularly varying functions with index $0$. Note also that a regularly varying function with index $\alpha$ admits the representation $f(x)=x^\alpha \tilde{f}(x)$, where $\tilde{f}$ is a slowly varying function. Moreover, by the Karamata representation theorem \cite[Theorem 1.2]{S} (see also \cite[Theorem IV.2.2]{Kor04}), $f$ is slowly varying at infinity precisely if there is $x_0\ge a$ such that
\be\label{eq:a.3}
f(x)=\exp\big\{\eta(x)+\int_{x_0}^{x} \frac{\varepsilon(t)}{t}dt\big\},\quad x\ge x_0>0,
\ee 
where $\eta$ is a bounded measurable function on $(x_0,\infty)$ such that $\lim_{x\to \infty}\eta(x)=\eta_0$, and $\varepsilon$ is a continuous function satisfying $\lim_{x\to\infty}\varepsilon(x)=0$. Regularly varying functions can be characterized by their behavior under integration against powers. This is the content of the following Karamata's characterization theorem (see \cite[\S IV.5]{Kor04} and \cite[\S I.5.6 and I.6.1]{BGT}).

\begin{theorem}[Karamata]\label{th:karamata}
Let the function $f:[a,+\infty)\to\R_+$ be positive and locally integrable.
\begin{itemize}
\item[(i)] If there are numbers  $\gamma$ and $\alpha>-\gamma$ such that
\be\label{eq:a.4}
\int_a^x t^{\gamma-1}f(t)dt\sim \frac{x^\gamma}{\gamma+\alpha}f(x),\quad x\to \infty,
\ee
then $f$ is regularly varying with index $\alpha$.

The same is true if there are numbers $\gamma$ and $\alpha<-\gamma$ such that
\be\label{eq:a.5}
\int_x^\infty t^{\gamma-1}f(t)dt\sim -\frac{x^\gamma}{\gamma+\alpha}f(x),\quad x\to \infty,
\ee

\item[(ii)] Conversely, if $f$ varies slowly at $\infty$, then
\be\label{eq:a.4B}
\frac{1}{f(x)}\int_a^x t^{\alpha}f(t)dt\sim \frac{x^{\alpha+1}}{\alpha+1},\quad x\to \infty,\quad (\alpha>-1)
\ee
and
\be\label{eq:a.5B}
\frac{1}{f(x)}\int_x^\infty t^{\alpha}f(t)dt\sim -\frac{x^{\alpha+1}}{\alpha+1},\quad x\to \infty,\quad (\alpha<-1).
\ee
The result remains true for $\alpha=-1$ in the sense that the integrals in the lefthand side of \eqref{eq:a.4B}, \eqref{eq:a.5B} tend to $\infty$. Moreover, in this case $\int_a^x\frac{f(t)}{t}dt$ is a slowly varying function. 
\end{itemize}
\end{theorem}  

Also we need the notion and some properties of positively increasing functions (see, e.g., \cite{bks, Rog}).

\begin{definition}\label{def:osv}
Let the function $f:\R_+\to \R_+$ be nondecreasing. The function $f$ is called {\em positively increasing at $\infty$} if there exists $t\in (0,1)$ such that
\begin{equation}\label{eq:osv}
S_\infty(t)\neq 1,\qquad S_\infty(t):=\limsup_{x\to \infty}\frac{f(xt)}{f(x)}.
\end{equation}

The function $f$ is called positively increasing at $0$ if $\tilde{f}(x)=1/f(1/x)$ is positively increasing at $\infty$, or equivalently, if there exists $t\in (0,1)$ such that
\begin{equation}\label{eq:osv_0}
S_0(t)\neq 1,\qquad S_0(t):=\limsup_{x\to 0}\frac{f(xt)}{f(x)}.
\end{equation}

\end{definition}

Note that (see, e.g., \cite{Ben87}) the function $S_f^i$, $i\in\{0,\infty\}$, is increasing and maps $(0,1)$ into $[0,1]$. Moreover, it is submultiplicative
\[  
S_f^i(t_1t_2)\le S_f^i(t_1)S_f^i(t_2),
\]
and hence either $S_f^i\equiv 1$ on $(0,1)$ or $S_f^i(t) \to 0$ as $t\downarrow 0$. The next result can be found in \cite{bks, Par03, Rog}.

\begin{lemma}\label{lem:osv}
Let the function $f:(0,1)\to \R_+$ be nondecreasing and bounded. Let also $f(0)=0$. The following are equivalent:
\begin{enumerate}
\item[(i)] $f$ is positively increasing at $0$,
\item[(ii)] there is $C\in (0,1)$ and $t\in (0,1)$ such that %for all $x\in$
\begin{equation}\label{eq:p1}
f(xt)\le Cf(x),\quad x\in (0,1),
 \end{equation}
 \item[(iii)] for each $t\in (0,1)$ there is $C\in (0,1)$  such that %for all $x\in$
\begin{equation}\label{eq:p1B}
f(xt)\le Cf(x),\quad x\in (0,1),
 \end{equation}
 \item[(iv)] there is  $C,\beta>0$ such that for all $t\in (0,1)$
 \begin{equation}\label{eq:p2}
 f(xt)\le C t^\beta f(x),\quad x\in (0,1),
 \end{equation}
 \item[(v)] there are no sequences $a_b, b_n$ such that $0<a_n<b_b\le 1$ and
 \begin{equation}\label{eq:pyat}
\lim_{n\to \infty} \frac{a_n}{b_n}=0,\qquad \lim_{n\to \infty} \frac{f(a_n)}{f(b_n)}=1.
 \end{equation}
\end{enumerate}
\end{lemma}

\begin{remark}\label{rem:3.2}
A few remarks are in order:
\begin{itemize}
\item[(i)]
Clearly, the property of $f$ to be a positively increasing function at $0$ depends only on a local behavior of $f$ at $0$. Therefore, without loss of generality one can consider \eqref{eq:p1} and \eqref{eq:p2} on an arbitrary subinterval $x\in (0,\varepsilon)$, $\varepsilon\le1$.
\item[(ii)] Clearly, one can reformulate Lemma \ref{lem:osv} for positively increasing at infinity functions .
\item[(iii)] The class of positively increasing at $\infty$ functions contains as proper subclasses the class of all increasing functions that varies regularly
with index $\alpha>0$ and the class of all rapidly varying functions. However, it contains no function from the class of slowly varying
functions.
\end{itemize}
\end{remark}

\section{The Liouville transformation}\label{ap:LT}
In this subsection we present some results from \cite[\S 14]{KK2}.

Assume that the spectral problem 
\be\label{eq:sp23}
-y''+q(x)y=\lambda\, r(x)y,\quad x\in(0,b);\quad y'(0)=0,
\ee
is in the limit point case at $b$.  We shall also assume that \eqref{eq:sp23} has nonnegative spectrum. The latter is equivalent to the fact that the solution $c(x,0)$ of \eqref{eq:sp23} is positive on $(0,b)$. 
Then (see \cite[\S 14]{KK2}), we set
 \be\label{eq:r}
 \xi:=\xi(x)=\int_0^x \frac{dt}{c_0^2(t)},\quad B:=\xi(b),\quad \ti{r}(\xi)=r(x)c_0^4(x); \quad c_0(x):=c(x,0).
 \ee 
 Consider the following spectral problem
 \be\label{eq:sp24}
-\frac{d^2f}{d\xi^2}=\lambda\, \ti{r}(\xi)f,\quad \xi\in(0,B); \quad f'(0)=0.
\ee

Firstly, notice that $\tilde{r}\in L^1_{\loc}[0,B)$. Indeed, using \eqref{eq:r}, for all $\xi<B$ we get
\[
\int_0^\xi \ti{r}(\mu)d\mu=\int_0^{x} r(t)c_0^4(t)\frac{dt}{c_0^2(t)}=\int_0^{x} c_0^2(t)r(t)dt<\infty.
\]
%where $x(\xi)$ is the inverse of $\xi$ defined by \eqref{eq:r}.  
Hence the above definition is correct. 

Next define the map $U:L^2_{r}(0,b)\to L^2_{\ti{r}}(0,B)$ as follows
\be\label{eq:U}
U:v(x)\to u(\xi):= \frac{1}{c_0(x)}v(x).
\ee

Let us show that $U$ is isometric:
\[
\|Uv\|^2_{L^2_{\ti{r}}}=\int_0^B |u(\xi)|^2\ti{r}(\xi)d\xi=\int_0^b \Big| \frac{v(x)}{c_0(x)}\Big|^2 r(x)c_0^4(x) \frac{dx}{c_0^2(x)}=\int_0^b|v|^2r(x)dx=\|v\|^2_{L^2_{r}}.
\]
Furthermore, it is not difficult to check that $\ti{y}:=Uy$ solves equation \eqref{eq:sp24}
 if $y$ is a solution of equation \eqref{eq:sp23}. Indeed, this is immediate from the following representation of \eqref{eq:sp23} (cf. \cite[\S 14]{KK2})
\[
-c_0^2(x)\frac{d}{dx}\Big(c_0^2(x)\frac{d}{dx}\frac{y}{c_0(x)}\Big)=\lambda\, c_0^4(x)\, \frac{y}{c_0(x)}.
\]
Finally, let us show that $\ti{c}(\xi,z)=Uc(z)$ and $\ti{s}(\xi,z)=Uc(z)$ is a fundamental system of solution of equation \eqref{eq:sp24}. The latter immediately follows from relations
%Therefore %\footnote{In \cite{bg}, the description of boundary condition is obtained by using a different argument}, 
\begin{eqnarray*}
&(Uy)(0)=\frac{y(0)}{c_0(0)}=y(0),\\
 &(Uy)'(0)=\Big(c_0^2(x)\frac{d}{dx}\frac{y(x)}{c_0(x)}\Big)\Big|_{x=0}=y'(0)c_0(0)-y(0)c_0'(0)=y'(0).
\end{eqnarray*}
%To see this it suffices to note that 

Finally, let us note that $B<\infty$ if and only if $\lambda=0$ is the eigenvalue of \eqref{eq:sp23}. Moreover, 
\[
\int_0^B \xi^2\ti{r}(\xi)d\xi=\infty
\]
due to the assumption that \eqref{eq:sp23} is limit point at $b$.
 
Thus we arrived at the following result.
 \begin{proposition}\label{prop:LT}
 Let the spectral problems \eqref{eq:sp23} and \eqref{eq:sp24} be connected via \eqref{eq:r}. Let also $m$ and $\tilde{m}$ be the $m$-functions associated with the problems \eqref{eq:sp23} and $\eqref{eq:sp24}$, respectively. Then 
 \[
 m(\lambda)=\ti{m}(\lambda).
 \]
 Moreover, the $m$-functions corresponding to the Dirichlet problems also coincide.
 \end{proposition}

\noindent
{\bf Acknowledgments.}
The author is grateful to Illya Karabash, Mark Malamud and Gerald Teschl for numerous helpful discussions and hints with respect to the literature. 
  
 %\begin{remark}
%Assume that $r(x)=p(x)|x|^{\alpha}$, where $\alpha>-1$ and $p(x)\asymp c_+$ as $x\to \infty$. It was shown in \cite[Theorem %1.3]{KKM09} that in this case $0$ is a regular critical point for $H$. Somehow, condition (i) in Lemma 5.1 is close to this assumption %(it is possible to prove this at least under the additional smoothness assumptions on $p$). Therefore, Corollary 5.1 is "almost" a %criterion for the regularity of the critical point $0$!
%\end{remark}
%%%%%%%%%%%%%%%%%%%%%%%%%%%%%%%%%%%%%%%%%%%%%%%%%%%%%%%%%%%%%%%%%%%%%%%%%%%%
%%%%%%%%%%%%%%%%%%%%%%%%%%%%%%%%%%%%%%%%%%%%%%%%%%%%%%%%%%%%%%%%%%%%%%%%%%%%

\end{document}